\documentclass{article}

\usepackage{microtype}
\usepackage{graphicx}
\usepackage{subcaption}
\usepackage{booktabs} 

\usepackage{amsmath,amsfonts,bm}

\def\eqref#1{equation~\ref{#1}}

\def\1{\bm{1}}

\def\mI{{\bm{I}}}

\def\mO{{\bm{O}}}

\def\mQ{{\bm{Q}}}

\def\mU{{\bm{U}}}

\DeclareMathAlphabet{\mathsfit}{\encodingdefault}{\sfdefault}{m}{sl}
\SetMathAlphabet{\mathsfit}{bold}{\encodingdefault}{\sfdefault}{bx}{n}

\def\gB{{\mathcal{B}}}
\def\gC{{\mathcal{C}}}

\def\gE{{\mathcal{E}}}

\def\gN{{\mathcal{N}}}
\def\gO{{\mathcal{O}}}

\def\gS{{\mathcal{S}}}

\def\gX{{\mathcal{X}}}
\def\gY{{\mathcal{Y}}}

\newcommand{\R}{\mathbb{R}}

\DeclareMathOperator*{\argmin}{arg\,min}

\DeclareMathOperator{\sign}{sign}

\usepackage{hyperref}

\usepackage[preprint]{icml2026}

\usepackage{amsmath}
\usepackage{amssymb}
\usepackage{mathtools}
\usepackage{amsthm}
\usepackage{enumitem}

\usepackage[capitalize,noabbrev]{cleveref}

\theoremstyle{plain}
\newtheorem{theorem}{Theorem}[section]

\newtheorem{lemma}[theorem]{Lemma}
\newtheorem{corollary}[theorem]{Corollary}
\newtheorem{fact}[theorem]{Fact}
\theoremstyle{definition}
\newtheorem{definition}[theorem]{Definition}
\newtheorem{assumption}[theorem]{Assumption}
\newtheorem{example}[theorem]{Example}
\theoremstyle{remark}
\newtheorem{remark}[theorem]{Remark}

\usepackage[textsize=tiny]{todonotes}

\crefname{equation}{}{}
\Crefname{equation}{}{}
\usepackage{nicefrac}

\newcommand{\brox}[3]{\operatorname{brox}^{#1}_{#2}\!\left(#3\right)}
\newcommand{\eqdef}{\mathrel{:=}}

\DeclareMathOperator{\dom}{dom}
\DeclareMathOperator{\Proj}{Proj}

\DeclareMathOperator{\bdry}{bdry}

\newcommand{\ProjOn}[2]{\Proj_{#1}\!\left(#2\right)}

\newcommand{\prox}[2]{\operatorname{prox}_{#1}\!\left(#2\right)}

\newcommand{\rbrac}[1]{\left( #1 \right)}
\newcommand{\sbrac}[1]{\left[ #1 \right]}
\newcommand{\cbrac}[1]{\left\{ #1 \right\}}

\newcommand{\inner}[2]{\left\langle #1,\, #2 \right\rangle}
\newcommand{\norm}[1]{\| #1 \|}
\newcommand{\abss}[1]{\left\lvert #1 \right\rvert}

\newcommand{\dist}[2]{\operatorname{dist}(#1,#2)}
\newcommand{\distsq}[2]{\operatorname{dist}^2(#1,#2)}
\newcommand{\distpow}[3]{\operatorname{dist}^{#1}(#2,#3)}

\newcommand{\Exp}[1]{\mathbb{E}\!\left[ #1 \right]}

\usepackage{xcolor}

\definecolor{darkgreen}{RGB}{0,100,0}
\newcommand{\BPM}{\textcolor{darkgreen}{BPM}}
\newcommand{\PPM}{\textcolor{darkgreen}{PPM}}
\newcommand{\SPPM}{\textcolor{darkgreen}{SPPM}}
\newcommand{\TRPPM}{\textcolor{darkgreen}{TRPPM}}

\newcommand{\algname}[1]{\textcolor{darkgreen}{#1}}

\definecolor{dgreen}{RGB}{34,139,34}
\definecolor{dred}{RGB}{178,34,34}

\newcommand{\greencheck}{\textcolor{dgreen}{$\surd$}}
\newcommand{\redcross}{\textcolor{dred}{$\times$}}
\usepackage{multirow}

\usepackage[most]{tcolorbox}
\newtcolorbox{commentbox}{
  colback=orange!15,
  colframe=orange!80!black,
  boxrule=0.8pt,
  arc=3pt,
  left=6pt, right=6pt, top=6pt, bottom=6pt
}

\icmltitlerunning{Stabilized Proximal Point Method}

\begin{document}

\twocolumn[
  \icmltitle{Stabilized Proximal Point Method via Trust Region Control}

  \begin{icmlauthorlist}
    \icmlauthor{Hanmin Li}{aff}
    \icmlauthor{Kaja Gruntkowska}{aff}
    \icmlauthor{Peter Richt\'arik}{aff}
  \end{icmlauthorlist}

  \icmlaffiliation{aff}{Center of Excellence for Generative AI, KAUST, Thuwal, Saudi Arabia}

  \icmlcorrespondingauthor{Hanmin Li}{hanmin.li@kaust.edu.sa}

  \icmlkeywords{Optimization}

  \vskip 0.3in
]

\printAffiliationsAndNotice{}
\begin{abstract}
   The Proximal Point Method (\PPM) \cite{rockafellar1976monotone} is a fundamental tool for nonsmooth convex optimization. 
   However, its convergence is not linear under general convexity in the absence of strong convexity or other structural assumptions.
   To address this limitation, we study a trust-region stabilized proximal point scheme in which each proximal update is computed over a localized feasible region. 
   We show that this simple stabilization enforces non-vanishing steps and yields a linear decrease in objective values outside any prescribed neighborhood, without assuming smoothness or strong convexity. 
   Our analysis identifies a displacement condition as the key driver of linear descent and provides two complementary parameter regimes to guarantee it: fixing the trust-region radius and choosing the regularization properly, or fixing the regularization and selecting radii via a uniform displacement lower bound. 
   We further give explicit characterization of the linear regime conditions respectively, and prove that the trust-region is redundant under strong convexity, 
   Finally, we establish an exact equivalence with the Broximal Point Method (\BPM) \cite{gruntkowska2025ball} in the active constraint regime.
\end{abstract}

\section{Introduction}
Optimization provides a general framework for formulating and solving learning and inference problems across machine learning, signal processing, and statistics. Model training is typically cast as the minimization of an empirical risk or regularized objective, often of the form
\begin{align*}
   \label{eq:minimization-problem}
  \min_{x \in \R^d}\; f(x),
\end{align*}
where $x$ represents model parameters and $f$ encodes a loss function.
Within this broad class of problems, convex optimization plays a key role, as convexity allows local information to characterize global optimality and enables algorithms with strong guarantees.
As such, it provides a setting in which the interaction between optimization, statistics, and learning can be understood in a principled and mathematically complete way.

Throughout this work, we therefore focus on objective functions $f: \R^d \mapsto \R \cup \cbrac{+\infty}$ that are proper (which means that $\dom f \eqdef \cbrac{x \in \R^d: f(x) < +\infty}$ is nonempty), closed, convex and have at least one minimizer.

In this setting, \emph{proximal operators} constitute a cornerstone of modern optimization theory and practice.
First introduced by \citet{rockafellar1976monotone}, they underlie a broad class of algorithmic ideas in machine learning, often appearing implicitly even when not explicitly recognized.
At the core of these developments lies the proximal operator itself.
\begin{definition}
  The proximal operator with parameter $\gamma > 0$ associated with a function $f: \R^d \mapsto \R \cup \cbrac{+\infty}$ is 
  \begin{align*}
    \prox{\gamma f}{x} \eqdef \argmin_{z \in \R^d} \cbrac{f(z) + \frac{1}{2\gamma}\norm{z - x}^2}.
  \end{align*}
\end{definition}
Although proximal operators were originally introduced in the context of proximal algorithms, their significance extends far beyond this initial motivation.
They have found widespread practical use not only as theoretical constructs, but also as guiding principles for optimizer design, splitting methods, generalized projection viewpoints, and implicit regularization phenomena in modern optimization.
Proximal updates arise, explicitly or implicitly, in contemporary training pipelines, including accelerated methods, local training procedures, and adaptive optimizers such as \algname{AdamW}.
We provide a more detailed discussion of these connections in \Cref{sec:prox-review}.

Beyond serving as building blocks for more complex algorithms, one of the most classical and fundamental applications of proximal operators is their iteration as a stand-alone optimization method -- the original setting in which they were introduced.
This leads to the Proximal Point Method (\PPM) \cite{moreau1965proximite, martinet1970breve}, which iterates
\begin{align}
  x_{k+1} = \prox{\gamma_k f}{x_k}. \tag{\PPM}
\end{align}
Here, $\gamma_k \geq 0$ is commonly referred to as the \emph{step size} at iteration $k$.
This update is well defined whenever $f$ is proper, closed and convex, in which case the proximal operator is single valued \cite{bauschke2020correction,beck2017first}.
In practice, each iteration amounts to solving a proximal subproblem, whose computational cost depends on the structure of $f$ and is often handled via an inner optimization subroutine.

When $f$ is differentiable, the optimality condition of the proximal subproblem implies that, for any $\gamma > 0$, 
\begin{align*}
   z = \prox{\gamma f}{x} \;\Longleftrightarrow\; z + \gamma\nabla f(z) = x,
\end{align*}
As a result, the \PPM\ update can be equivalently written as 
\begin{align*}
   x_{k+1} = x_k - \gamma_k \nabla f(x_{k+1}).
\end{align*}
In this form, \PPM\ is closely related to Gradient Descent (\algname{GD}), except that the gradient is evaluated at the future iterate $x_{k+1}$, which makes the update implicit.
Such implicit schemes are known to exhibit better stability properties compared to their explicit counterparts \cite{ryu2016stochastic}.
As a result, proximal methods, unlike classical gradient-based algorithms, often exhibit notable robustness to inaccuracies in learning rate specifications \cite{asi2019stochastic}.

Beyond stability, an important advantage of proximal methods over widely used gradient-based approaches such as \algname{GD} and \algname{SGD} \cite{robbins1951stochastic,nemirovski2009robust,gower2019sgd} is that their convergence guarantees typically do not rely on smoothness assumptions on the objective function \cite{richtarik2024unified}.
This property makes them particularly attractive in modern deep learning applications, where loss functions are often nonsmooth or only weakly regularized \cite{zhang2020gradient}.

It is well known that under additional conditions, such as strong convexity or an error bound, \PPM\ enjoys a fast \emph{linear} convergence rate.
In the absence of such assumptions, however, no uniform linear rate can be guaranteed, and \PPM\ may converge only sublinearly in the worst case.
This phenomenon occurs even for simple one-dimensional problems.
\begin{example}
   \label{exp:1dconvex-sublinear}
   Consider the convex function 
   \begin{align*}
      f(x) \eqdef \frac{1}{4}x^4,
   \end{align*}
   and run \PPM\ with a constant step size $\gamma > 0$.
   The optimality condition of the proximal subproblem yields $x_k = x_{k+1} + \gamma x_{k+1}^3$, which implies that $x_k \to 0$ at a polynomial rate.
   In particular, $x_k = \Theta(k^{-1/2})$, and hence $f(x_k) = \Theta(k^{-2})$.
\end{example}
For the classical {\PPM}, this slow tail behavior is a fundamental limitation. While linear convergence is known to be achievable for general convex objectives without additional assumptions \citep{gruntkowska2025ball,gruntkowska2025non}, {\PPM} lacks such guarantees. This motivates the question of whether {\PPM} itself can be modified to achieve linear convergence for any proper, closed, and convex objective. In this work, we answer this question affirmatively.

To this end, we propose a novel trust-region stabilization of {\PPM} and show that it provides the key ingredient for achieving a linear decrease in the objective value, as long as the iterates remain outside a prescribed neighborhood of the minimizer set.
Concretely, we introduce the trust-region Proximal Point Method (\TRPPM), defined by the iteration
\begin{align}
   \label{eq:trppm-def}
   x_{k+1} = \argmin_{z \in \gB_{t_k}(x_k)} \cbrac{f(z) + \frac{\lambda_k}{2}\norm{z - x_k}^2} \tag{\TRPPM},
\end{align}
where $\gB_{t_k}(x_k) \eqdef \cbrac{z:\norm{z - x_k} \leq t_k}$ denotes the trust-region centered at $x_k$ with radius $t_k$, and the parameter $\lambda_k \geq 0$ controls the strength of the quadratic regularization.
In other words, {\TRPPM} coincides with the update rule of the classical {\PPM}, but with the optimization restricted to the local neighborhood $\gB_{t_k}(x_k)$ rather than the entire space~$\R^d$.
This seemingly minor modification plays a central role in both the algorithmic behavior and the theoretical properties of the method. In the remainder of the paper, we develop several complementary perspectives that illuminate the motivation for introducing this trust-region constraint and clarify how it interacts with proximal regularization.

\section{Contributions}

We summarize our contributions as follows:
\begin{itemize}
   \item \textbf{Trust-region stabilization yields linear descent under pure convexity.}
   We show that, for any proper, closed, and convex objective with a nonempty solution set, {\TRPPM} achieves a \emph{linear decrease} in the function value suboptimalities as long as the iterates remain outside a prescribed neighborhood of the solution set.
   This result requires neither strong convexity nor smoothness, and is achieved by enforcing a uniform lower bound on the step size away from the target neighborhood. 

    \item \textbf{Two complementary mechanisms for entering the linear regime.}
    We identify two distinct yet complementary ways to guarantee linear descent.
    In the first regime, the trust-region radius is fixed ($t_k \equiv t > 0$), and the regularization parameter is chosen such that $\lambda_k \le \lambda_k^\star$ (defined in \Cref{lemma:existence}), so that the proximal displacement $\norm{x_k - \prox{\nicefrac{f}{\lambda_k}}{x_k}}$ exceeds $t$, ensuring that the trust-region constraint is active at every non-tail iterate.
    In the second regime, the regularization parameter is fixed ($\lambda_k \equiv \lambda$), and the trust-region radii are selected according to a uniform lower bound $m_f(\varepsilon,\lambda)$ (defined in \Cref{def:m-quantity}) on the proximal displacement over points away from the solution set.

   \item \textbf{Explicit admissible parameter choices under weak sharp minima.}
   We provide general estimates for admissible values of the regularization parameter $\lambda_k$ and characterize the role of the displacement lower bound $m_f(\varepsilon, \lambda)$ in controlling the linear descent rate.
   Under the classical weak sharp minima condition ($q=1$), we further derive explicit sufficient bounds on $\lambda_k$ that guarantee entry into the linear regime.
   These results turn the existence of a critical regularization threshold and a displacement-based rate constant into verifiable rules depending only on the function value gap and the trust-region radius.

   \item \textbf{Trust-region redundancy under strong convexity.}
   We show that when the objective is strongly convex, {\TRPPM} achieves linear convergence even without enforcing a trust-region constraint, in agreement with classical theory.
   In this regime, the contraction factor can be controlled solely through the regularization parameter $\lambda_k$, explaining precisely why trust-region stabilization is necessary in the merely convex setting but becomes redundant under strong convexity.

   \item \textbf{A bridge between \PPM\ and \BPM.}
   We explain the source of {\TRPPM}'s improvement over the classical {\PPM} by showing that \TRPPM\ continuously interpolates between {\PPM} and the Broximal Point Method (\BPM) \citep{gruntkowska2025ball} (see \Cref{sec:prox-review}) through the joint choice of the parameters $(\lambda_k, t_k)$.
   In particular, when the unconstrained minimizer of the regularized model lies outside the trust-region, the \TRPPM\ update coincides exactly with the corresponding \BPM\ step.
   This equivalence shows that quadratic regularization improves the conditioning of the subproblems while leaving the iterates unchanged in the active constraint regime.
\end{itemize}

The rest of the paper is organized as follows.
In \Cref{sec:prox-review}, we present a review of topics related to proximal operators and highlight their broad range of applications.
\Cref{sec:conv-analysis} introduces the main assumptions and presents the general convergence theory of {\TRPPM}. 
The two complementary regimes are studied separately: the case of a fixed trust-region radius is analyzed in \Cref{sec:fixt}, while the case of a fixed regularization parameter is treated in \Cref{sec:fixlambda}.
In \Cref{sec:quantity-m}, we analyze the uniform lower bound on the proximal displacement and illustrate its role through concrete examples.
An explanation of why the trust-region constraint becomes redundant in the strongly convex setting is given in \Cref{sec:stronglyconvex}.
Finally, we conclude in \Cref{sec:conclusion}.

The notation used throughout the paper is summarized in \Cref{sec:notation}, and the acceleration mechanism from the perspective of {\BPM} is discussed in \Cref{sec:mechanism}.

\begin{table}[t]
  \caption{Comparison of linear convergence guarantees for {\PPM}, {\TRPPM}, and {\BPM}. 
  Neighborhood means that the algorithm converges linearly outside any prescribed neighborhood.}
  \label{tab:linear-conv-comparison}
  \begin{center}
    \begin{small}
      \begin{sc}
        \begin{tabular}{lc}
          \toprule
          Method & Linear Convergence \\
          \midrule
          {\PPM}
            & \redcross \\[0.5ex]

          {\TRPPM} (fixed $t$)
            & \greencheck\  (neighborhood) \\[0.5ex]

          {\TRPPM} (fixed $\lambda$)
            & \greencheck\  (neighborhood) \\[0.5ex]

          {\BPM}
            & \greencheck \\
          \bottomrule
        \end{tabular}
      \end{sc}
    \end{small}
  \end{center}
  \vskip -0.1in
\end{table}

\section{The many facets of {\PPM}}
\label{sec:prox-review}

Since their introduction, proximal methods have appeared in a wide range of contexts across optimization and machine learning.
We review their most prominent applications and highlight the connections that motivate our work.

\paragraph{Optimizer Design.}
Proximal algorithms have had a lasting impact on the design of modern optimizers in machine learning.
Rather than minimizing the objective function directly, many algorithms can be interpreted as approximating a proximal step, often by linearizing the loss while retaining a proximal treatment of regularization or stabilizing terms.
A notable example is the widely used optimizer \algname{AdamW} \cite{loshchilov2018decoupled}, which admits a proximal interpretation \cite{zhuang2022understanding}.
From this perspective, the decoupled weight decay (“W”) corresponds to applying the proximal operator of a quadratic regularizer in conjunction with an adaptive gradient step on the loss.
More broadly, trust-region constraints, projection steps, and gradient clipping can be interpreted as heuristic approximations of proximal updates, in the sense that they impose implicit regularization by restricting the update to a bounded neighborhood of the current iterate \citep{rockafellar1976monotone, parikh2014proximal, nocedal2006numerical}.

\paragraph{Distributed Optimization.}
Motivated by the growing scale and decentralization of modern learning systems, the proximal point framework has also been extended to stochastic and distributed settings \cite{konevcny2016federated,mcmahan2017communication}.
This has led to the Stochastic Proximal Point Method (\SPPM) \cite{bertsekas2011incremental,khaled2023faster}, also known as \algname{FedProx} \cite{li2020federated} in the context of federated learning.
Proximal operators have proven useful in local training, a technique commonly used in federated learning.
This refers to the practice of allowing each participating client to perform several local optimization steps on its own data before a communication-intensive synchronization step is carried out.
From this viewpoint, local training can be understood as each client approximately solving a proximal subproblem centered at the current global iterate.
Variants of local training were already used in earlier work \citep{povey2014parallel,moritz2016sparknet}, before being popularized in federated learning by \citet{mcmahan2017federated}.
Proximal algorithms can also be used directly to formulate local training schemes, rather than merely to interpret them, as exemplified by \algname{ProxSkip} \citep{mishchenko2022proxskip}, which interleaves local gradient steps with occasional proximal updates to reduce communication while retaining proximal convergence guarantees.

\paragraph{Projections.}
It is well known that proximal operators are closely related to projection operators.
In particular, the proximal operator of a convex function can be interpreted as a generalized projection onto a certain level set of the objective \cite{li2024power}.
Another classical connection arises when the function in the proximal operator is chosen to be the indicator function of a set $\gC$.
In this case, the proximal operator reduces exactly to the Euclidean projection,
\begin{align*}
   \prox{\iota_{\gC}}{x} = \ProjOn{\gC}{x}.
\end{align*}
This identity shows that projection methods are a special case of proximal methods.
Projection algorithms were originally introduced for solving linear systems \cite{kaczmarz1993approximate}, and were later extended to the more general setting of convex feasibility problems \cite{combettes1997hilbertian}.
In a similar spirit, {\SPPM} can be viewed as a generalization of parallel projection methods.

It is known that parallel projection methods can be accelerated by extrapolation, both in practice and in theory \cite{combettes1997convex,necoara2019randomized}.
In this setting, extrapolation consists of taking an over-relaxed step beyond the averaged projection point along the direction from the previous iterate, a construction that is closely related to classical over-relaxation in fixed point iterations and proximal splitting methods \cite{iutzeler2019generic,condat2023proximal}.
More recently, extrapolation has been shown to be effective in the more general setting of {\SPPM}, leading to improved performance in both theory and practice \cite{li2024power,li2024convergence, anyszka2024tighter}.

\paragraph{Splitting Algorithms.}

Proximal splitting methods extend the proximal point principle to structured objectives where evaluating the full proximal map of the sum is computationally prohibitive, but the individual terms admit cheap proximal or gradient evaluations \cite{combettes2011proximal,condat2023proximal}.
A canonical setting is the composite problem
\begin{align*}
  \min_{x \in \R^d}\; \cbrac{f(x) + g(x)},
\end{align*}
where splitting replaces a single proximal step on $f+g$ by inexpensive operations that process $f$ and $g$ separately.
When one term is smooth and the other is proximable, this leads to forward-backward splitting (proximal gradient), along with its variants \cite{tseng2000modified}.
When both terms are nonsmooth but proximable, Douglas-Rachford splitting applies the two proximal maps in an alternating manner, which is closely related to the alternating direction method of multipliers \algname{ADMM} \cite{lions1979splitting,eckstein1992douglas,boyd2011distributed}.
More general models involving linear operators are naturally treated by primal-dual splitting schemes, which perform explicit primal-dual updates in a product space and avoid inner linear solves \cite{chambolle2011first,condat2013primal,vu2013splitting}.
A unifying viewpoint is that many splitting algorithms can be interpreted as preconditioned or relaxed versions of {\PPM} applied to a monotone inclusion, providing insight into step-size rules and over-relaxation \cite{bauschke2020correction,davis2017three,condat2023proximal}.

\paragraph{Acceleration.}
Acceleration is closely intertwined with the proximal point framework, both as a mechanism that can be interpreted through {\PPM} and as a property that {\PPM} itself can inherit.
On the one hand, Nesterov acceleration can be understood as an approximation of a proximal update.
As shown in \citet{ahn2022understanding}, Accelerated Gradient Descent (\algname{AGD}) arises by centering a proximal step at an extrapolated point and replacing the exact proximal subproblem by a linearization, i.e., 
\begin{align*}
   x_{k+1} &= \argmin_{z} \cbrac{\inner{\nabla f(y_k)}{z - y_k} + \frac{1}{2\alpha_k}\norm{z - y_k}^2} \\
   y_{k+1} &= x_k + \beta_k(x_{k+1} - x_k),
\end{align*}
which enjoys the classical $\gO\rbrac{\nicefrac{1}{K^2}}$ rate for convex objectives.
On the other hand, {\PPM} itself admits accelerated variants.
\citet{guler1992new} established that {\PPM} can achieve accelerated convergence for convex minimization despite its implicit and non-smooth nature.
Subsequent work showed that acceleration can be retained under inexact solution of proximal subproblems, leading to accelerated inexact and hybrid proximal frameworks \cite{he2012accelerated,monteiro2013accelerated}.
This viewpoint also motivates generic acceleration frameworks such as \algname{Catalyst} \cite{lin2015universal,lin2018catalyst}, which can be seen as an inexact accelerated {\PPM} applied on top of a base method.
More recently, acceleration has also been obtained by modifying the proximal geometry itself. 
In particular, replacing the quadratic regularizer in {\PPM} with a higher order one leads to higher-order proximal methods that achieve faster polynomial convergence rates for convex objectives \cite{nesterov2023inexact}.

\paragraph{Smoothing.}

The proximal operator is closely related to the Moreau envelope \cite{moreau1965proximite}.
It is well known that applying proximal algorithms to the original objective is equivalent to running gradient-based methods on the corresponding Moreau envelope \cite{ryu2016stochastic,li2024power}.
Importantly, this reformulation does not alter the solution set, since the minimizers of the original objective and its Moreau envelope coincide \cite{planiden2016strongly,planiden2019proximal}.
Beyond proximal minimization, the Moreau envelope has found applications in other areas, including personalized federated learning \cite{t2020personalized} and meta-learning \cite{mishchenko2023convergence}.

\paragraph{BPM and Trust Region Method.}

Another line of work rooted in the proximal framework is the Broximal Point Method (\BPM) \cite{gruntkowska2025ball, gruntkowska2025non}.
{\BPM} replaces the regularization used in the proximal subproblems with an explicit trust-region-type constraint that directly bounds the step size.
Specifically, the method relies on the broximal operator:
\begin{definition}
   \label{def:brox}
   The broximal operator with radius $t>0$ associated with a function $f:\R^d \mapsto \R \cup \cbrac{+\infty}$ is given by 
   \begin{align}
      \brox{t}{f}{x} \eqdef \argmin_{z \in \gB_t(x)} f(z),
   \end{align}
   where $\norm{\cdot}$ is the standard Euclidean norm.
\end{definition}

From a theoretical perspective, {\BPM} is known to enjoy fast linear convergence under the assumptions that the objective is proper, closed, and convex.
However, each iteration of {\BPM} requires solving a local optimization problem involving the original objective $f$ restricted to a neighborhood of the current iterate.
In the absence of additional structure, this local subproblem can be as challenging as the original global problem.
Recently, {\BPM} has been extended to settings in which the trust-region is defined by a non-Euclidean geometry rather than the standard Euclidean norm \cite{gruntkowska2025non}.
The broximal framework is closely related to trust-region methods, whose origins can be traced back to early work by \citet{levenberg1944method} and were later formalized by \citet{marquardt1963algorithm} in the context of nonlinear optimization.
The main idea behind these methods is to minimize, at each iteration $k$, a local model~$m_k$ of the objective~$f$ around the current iterate~$x_k$, subject to a trust-region constraint that enforces locality.
In the simplest Euclidean setting, this leads to updates of the form
\begin{align*}
  x_{k+1} = \argmin_{z \in \gB_{t_k}(x_k)} m_k(z),
\end{align*}
where $t_k > 0$ controls the radius of the trust-region.
The use of balls as trust-regions has been studied in several recent contexts.
Acceleration techniques have been combined with ball operators \cite{carmon2020acceleration}, and similar constrained formulations have been employed in problems involving minimization of maximum or worst case losses \cite{carmon2021thinking,asi2021stochastic}.
Across these works, the additional ball constraint provides explicit control over the step length, thereby preventing excessively large updates that may arise from ill-conditioned objectives, inexact subproblem solutions, or imperfect local models.

\section{Convergence Theory of {\TRPPM}}
\label{sec:conv-analysis}

Our proposed algorithm {\TRPPM} admits two complementary interpretations.
On the one hand, it can be understood as a trust-region variant of the standard {\PPM}.
On the other hand, it can be interpreted as an instance of approximate {\BPM} based on the local model
\begin{align*}
   m_k(z) \eqdef f(z) + \frac{\lambda}{2}\norm{z - x_k}^2,
\end{align*}
that is, a local model whose approximation error vanishes at the center.
By appropriate choices of the parameters $t_k$ and $\lambda_k$, {\TRPPM} interpolates between these two methods: setting $t_k = +\infty$ recovers \PPM, while setting $\lambda_k = 0$ yields \BPM.
To better understand the interplay between~$\lambda_k$ and~$t_k$, we consider two complementary viewpoints in which $\lambda_k$ is fixed while $t_k$ varies, or $t_k$ is fixed while $\lambda_k$ varies.
Throughout the analysis, we assume that our objective function satisfies the following assumption.
\begin{assumption}
   \label{assmp:pcc}
   The function $f: \R^d \mapsto \R \cup \cbrac{+\infty}$ is proper, closed and convex, and has a non-empty set of minimizers $\gX_\star^f$.
\end{assumption}

\subsection{Fixing $t_k \equiv t$}
\label{sec:fixt}

Before presenting the convergence analysis, we introduce a key quantity that appears in the subsequent arguments.
\begin{definition}
   \label{def:displacement-func}
   Let $f:\R^d \mapsto \R \cup \cbrac{+\infty}$ satisfy \Cref{assmp:pcc}.
   For $\lambda > 0$ and $x \in \dom f$, we define the \emph{displacement function} $\phi(\lambda, x)$ by
   \begin{align*}
      \phi(\lambda, x) \eqdef \norm{x - \prox{\nicefrac{f}{\lambda}}{x}}.
   \end{align*}
   In the degenerate case $\lambda = 0$, we define $\prox{\nicefrac{f}{\lambda}}{x} \eqdef \ProjOn{\gX_\star^f}{x}$, so that
   \begin{align*}
      \phi(0, x) = \dist{x}{\gX_\star^f}.
   \end{align*}
\end{definition}
The function $\phi$ measures the distance between a point $x \in \R^d$ and its proximal minimizer, that is, the minimizer of the regularized objective.
As such, it provides a convenient way to quantify how far the proximal update would move from the current iterate.

The following lemma characterizes the behavior of $\phi(\lambda, x)$ as a function of the regularization parameter $\lambda$.
\begin{lemma}[Existence of $\lambda$]
   \label{lemma:existence}
   Let $f : \R^d \to \R \cup \cbrac{+\infty}$ satisfy \Cref{assmp:pcc}, and consider the update \eqref{eq:trppm-def} with a fixed radius $t_k \equiv t > 0$.
   For any iterate $x_k$ such that $\dist{x_k}{\gX_\star^f} > t$, there exists $\lambda_k^\star > 0$ such that
   \begin{align*}
      \phi(\lambda_k^\star, x_k) = t.
   \end{align*}
   Moreover, for any $\lambda_k \leq \lambda_k^\star$, it holds that
   \begin{align*}
      \norm{x_k - \prox{\nicefrac{f}{\lambda_k}}{x_k}} \geq t.
   \end{align*}
\end{lemma}
The above lemma guarantees that, for any iterate outside the $t$-neighborhood of $\gX_\star^f$, there exists a value $\lambda_k^\star$ for which the displacement matches the trust-region radius.
Further properties of the displacement function are provided in \Cref{sec:displacement-lambda}.
We now use this characterization to establish the convergence behavior of {\TRPPM}.
\begin{theorem}
   \label{thm:linear-conv}
   Let $f:\R^d \mapsto \R \cup \cbrac{+\infty}$ satisfy \Cref{assmp:pcc} and let $\{x_k\}$ be the iterates of \Cref{eq:trppm-def} run with a fixed trust-region radius $t_k \equiv t > 0$,
   where the sequence $\cbrac{\lambda_k}$ satisfies $0 < \lambda_k \leq \lambda_k^\star$, with $\lambda_k^\star$ defined in \Cref{lemma:existence}.
   Then, for any iterate $K$ such that $\dist{x_K}{\gX_\star^f} > t$ (i.e., outside the $t$-neighborhood of the solution set), the following descent inequality holds:
   \begin{align*}
      f(x_K) - f_{\inf} \leq \rbrac{\frac{1}{1 + \nicefrac{t}{d_0}}}^K \rbrac{f(x_0) - f_{\inf}},
   \end{align*}
   where $d_0 \eqdef \dist{x_0}{\gX_\star^f}$.
\end{theorem}
The result shows that when the trust-region radius $t$ is fixed, the convergence of \TRPPM\ can be explicitly regulated through a suitable choice of the parameter sequence ${\lambda_k}$, ensuring linear convergence of the algorithm.
In general, the sequence $\{\lambda_k\}$ admits the following upper bound.
\begin{lemma}
   \label{lemma:roughestimate}
   Let assumptions of \Cref{thm:linear-conv} hold and suppose that $\dist{x_k}{\gX_\star^f} > t$ and $\norm{x_k - \prox{\nicefrac{f}{\lambda_k^\star}}{x_k}} \geq t$ for some $k$.
   Then
    \begin{align*}
      \lambda_k^\star \leq \frac{2\rbrac{f(x_k) - f_{\inf}}}{t^2}.
   \end{align*}
\end{lemma}
Without additional assumptions on the geometry of $f$, it is in general impossible to obtain a more explicit characterization of the optimal parameter $\lambda_k^\star$.
However, such bounds can be derived under suitable growth conditions on the objective.
We illustrate this point by introducing a generalized version of the \emph{weak sharp minima condition}.
\begin{assumption}[Generalized Weak Sharp Minima]
   \label{assmp:sharp-minima}
   Let $f: \R^d \mapsto \R \cup \cbrac{+\infty}$ satisfy \Cref{assmp:pcc}.
   We say that $f$ has \emph{weak sharp minima of order $q \geq 1$ with constant $\alpha > 0$} if
   \begin{align*}
      f(x) - f_{\inf} \geq \alpha \cdot \distpow{q}{x}{\gX_\star^f}
      \qquad \forall x \in \R^d .
   \end{align*}
\end{assumption}
When $q = 1$, this condition reduces to the classical weak sharp minima assumption \cite{burke1993weak,burke2002weak}, which quantifies a linear growth of the objective away from the solution set.
Notably, it is strictly weaker than strong convexity, but nevertheless enables us to explicitly bound the parameter $\lambda_k^\star$, as illustrated by the following lemma.
For simplicity, we focus on the case $q=1$; extensions to higher orders are deferred to the appendix.
\begin{lemma}
   \label{lemma:simplifiedlambda}
   Let $f:\R^d \mapsto \R \cup \cbrac{+\infty}$ satisfy \Cref{assmp:pcc} and \Cref{assmp:sharp-minima} with parameters $\alpha > 0$ and $q = 1$.
   Let $\{x_k\}$ be the iterates of \Cref{eq:trppm-def} run with $t_k \equiv t > 0$, and suppose that the current iterate lies outside the $t$-neighborhood of the solution set, i.e., $\dist{x_k}{\gX_\star^f} > t$.
   If the parameters $\lambda_k$ satisfy
   \begin{align*}
      0 < \lambda_k \leq \frac{2\alpha^2}{f(x_k) - f_{\inf} + \alpha t},
   \end{align*}
   then it holds that
   \begin{align*}
      \norm{x_k - \prox{\nicefrac{f}{\lambda_k}}{x_k}} \geq t.
   \end{align*}
\end{lemma}
Combining this bound with \Cref{thm:linear-conv}, we obtain the following immediate corollary.
\begin{corollary}
   \label{col:thm:linear-conv}
   Let $f:\R^d \mapsto \R \cup \cbrac{+\infty}$ satisfy \Cref{assmp:pcc} and \Cref{assmp:sharp-minima} with parameters $\alpha > 0$ and $q = 1$.
   Let $\{x_k\}$ be the iterates of \Cref{eq:trppm-def} run with $t_k \equiv t > 0$ and parameter sequence $\{\lambda_k\}$ such that
   \begin{align*}
      0 < \lambda_k \leq \frac{2\alpha^2}{f(x_k) - f_{\inf} + \alpha t}.
   \end{align*}
   Then the following descent inequality holds:
   \begin{align*}
      f(x_K) - f_{\inf}
      \leq
      \rbrac{\frac{1}{1 + \nicefrac{t}{d_0}}}^K
      \rbrac{f(x_0) - f_{\inf}}.
   \end{align*}
\end{corollary}
We conclude by comparing this result with the classical convergence guarantees for {\PPM}.
\begin{remark}[Comparison with {\PPM}]
    Under \Cref{assmp:pcc}, it is well known \citep[Theorem 10.28]{beck2017first} that the iterates of {\PPM} with a constant step size $\lambda_k \equiv \lambda$ satisfy
    \begin{align}
    \label{eq:compconv1}
        f(x_K) - f_{\inf} \leq \frac{d_0^2}{2\sum_{k=0}^{K-1} \lambda_k^{-1}}= \frac{d_0^2}{2K\lambda^{-1}},
    \end{align}
    yielding an $\mathcal{O}(1/K)$ convergence rate in function values \citep{rockafellar1976monotone}.
    In particular, \Cref{eq:compconv1} indicates that {\PPM} does not admit a uniform linear convergence rate without additional assumptions, and forcing a linear decay via step size tuning would require $\lambda_k$ that shrink in a pathological manner (exponentially fast with respect to $K$), which in turn makes each proximal subproblem increasingly expensive to solve.
    In contrast, {\TRPPM} achieves linear convergence even with a fixed trust-region radius $t_k \equiv t > 0$, by choosing $\lambda_k$ within the admissible range.
    {\TRPPM} controls the per iteration effort by enforcing a trust-region constraint, thereby stabilizing the subproblems while still guaranteeing linear convergence.
\end{remark}

\subsection{Fixing $\lambda_k \equiv \lambda$}
\label{sec:fixlambda}

We now turn to the complementary regime in which the regularization parameter is held fixed, $\lambda_k \equiv \lambda$, while the trust-region radii ${t_k}$ are allowed to vary.
We begin by introducing a quantity that captures the minimal proximal displacement away from the solution set.
\begin{definition}
   \label{def:m-quantity}
   Let $f: \R^d \mapsto R \cup \cbrac{+\infty}$ satisfy \Cref{assmp:pcc} and fix any $x_\star^f \in \gX_\star^f$, $\varepsilon > 0$.
   We define 
   \begin{align*}
      m_f(\varepsilon, \lambda) \eqdef \min_{x \in \gS_\varepsilon} \norm{x - \prox{\nicefrac{f}{\lambda}}{x}},
   \end{align*}
   that is, the uniform lower bound of the displacement function $\phi(\lambda, x)$ over the set
   \begin{align*}
      \gS_\varepsilon \eqdef \gB_{\norm{x_0 - x_\star^f}}(x_\star^f) \cap \cbrac{x \in \R^d: \dist{x}{\gX_\star^f} \geq \varepsilon}.
   \end{align*}
\end{definition}
Although $\gS_\varepsilon$ depends on the initial point $x_0 \in \dom f$, we treat $x_0$ as fixed throughout the discussion to avoid notational clutter.

\begin{lemma}[Uniform lower bound]
   \label{lemma:uniform-lowerbound}
   Let $f: \R^d \mapsto \R \cup \cbrac{+\infty}$ satisfy \Cref{assmp:pcc}.
   Then
   \begin{align*}
      m_f(\varepsilon, \lambda) > 0.
   \end{align*}
\end{lemma}

Now, let us formalize the convergence guarantee.
\begin{theorem}
    \label{lemma:lnrate-fixed-lambda}
    Let $f:\R^d \mapsto \R \cup \cbrac{+\infty}$ satisfy \Cref{assmp:pcc} and let $\{x_k\}$ be the iterates of \Cref{eq:trppm-def} run with fixed $\lambda_k \equiv \lambda \geq 0$ and (possibly varying) trust-region radii $t_k \geq 0$.
   Then, for any $\varepsilon > 0$, as long as the iterates remain outside the $\varepsilon$-neighborhood of the solution set $\gX_\star^f$, one can choose
   \begin{align*}
      t_k = \theta \cdot m_f(\varepsilon, \lambda) > 0,
      \qquad
      \theta \in (0,1],
   \end{align*}
   where $m_f(\varepsilon, \lambda)$ is defined in \Cref{lemma:uniform-lowerbound}.
   Under this choice, the iterates satisfy the linear descent inequality
   \begin{align*}
      f(x_{k+1}) - f_{\inf}
      \leq
      \rbrac{\frac{1}{1 + \theta \cdot \nicefrac{m_f(\varepsilon, \lambda)}{d_0}}}
      \rbrac{f(x_k) - f_{\inf}} .
   \end{align*}
\end{theorem}

\begin{remark}
   In \citet[Corollary 8.2]{gruntkowska2025ball}, it is shown that under \Cref{assmp:pcc}, the iterates of {\BPM} satisfy
   \begin{align*}
      f(x_{k+1}) - f_{\inf} \leq \rbrac{1 + \frac{t_k}{d_0}}^{-1}\rbrac{f(x_k) - f_{\inf}}.
   \end{align*}
   Specializing \Cref{lemma:lnrate-fixed-lambda} to the case $\lambda = 0$ and $\theta = 1$ yields $m_f(\varepsilon,\lambda) = \varepsilon$. Hence, the two convergence results coincide when $t_k = \varepsilon$.
\end{remark}

\subsection{The quantity $m_f(\varepsilon, \lambda)$}
\label{sec:quantity-m}

To better understand the role of the quantity $m_f(\varepsilon, \lambda)$ in the convergence analysis, we now examine how it depends on the parameters $\varepsilon$ and $\lambda$. We begin by discussing several basic properties before illustrating them through explicit examples.

\begin{remark}[Edge case]
   When $\lambda = +\infty$, the proximal operator reduces to the identity mapping, $\prox{\nicefrac{f}{\lambda}}{x} \equiv x$.
   As a result, the displacement vanishes for all $x$, and hence $m_f(\varepsilon, \lambda)= 0$. 
\end{remark}

\begin{remark}[Monotonicity in $\varepsilon$ for fixed $\lambda$]
   The neighborhood radius $\varepsilon > 0$ directly influences the magnitude of the threshold $m_f(\varepsilon, \lambda)$.
   Intuitively, points that are further away from the solution set are expected to exhibit a larger displacement under the proximal mapping.
   This intuition can be made precise.
   For a fixed minimizer $x_\star^f\in\gX_\star^f$ and any $0 < \varepsilon_1 \leq \varepsilon_2$, we have $\gS_{\varepsilon_2} \subseteq \gS_{\varepsilon_1}$,
   which immediately implies the monotonicity property
   \begin{align*}
      m_f(\varepsilon_2, \lambda) \geq m_f(\varepsilon_1, \lambda).
   \end{align*}
\end{remark}

\begin{lemma}[Monotonicity]
   \label{lemma:undecided}
   Let $f:\R^d \mapsto \R \cup \cbrac{+\infty}$ satisfy \Cref{assmp:pcc}.
   For any fixed $\varepsilon > 0$, the function $m_f(\varepsilon, \lambda)$ is a non-increasing function of $\lambda$.
\end{lemma}

\begin{lemma}[Continuity]
   \label{lemma:cont-m}
   Let $f:\R^d \mapsto \R \cup \cbrac{+\infty}$ satisfy \Cref{assmp:pcc}.
   For any fixed $\varepsilon > 0$, $m_f(\varepsilon, \lambda)$ is a continuous function of $\lambda \geq 0$.
\end{lemma}
When the objective function $f$ admits a simple structure, the quantity $m_f(\varepsilon, \lambda)$ can often be computed in closed form.
\begin{example}[Indicator function]
   Let $f(x) = \iota_{\gC}(x)$, where $\gC \subseteq \R^d$ is a closed convex set.
   In this case, the proximal operator reduces to the Euclidean projection onto~$\gC$, namely,
   \begin{align*}
      \prox{\nicefrac{f}{\lambda}}{x} = \ProjOn{\gC}{x}.
   \end{align*}
   Consequently, the displacement function simplifies to
   \begin{align*}
      \phi(\lambda, x)
      = \norm{x - \ProjOn{\gC}{x}}
      = \dist{x}{\gC}.
   \end{align*}
   By definition of $\gS_\varepsilon$, it follows that $m_f(\varepsilon, \lambda) = \varepsilon$.
\end{example}

\begin{example}
   Let $f(x) = \abss{x}$. The proximal map for this function is
   \begin{align*}
      \prox{\nicefrac{f}{\lambda}}{x} = \sign(x) \cdot \max \cbrac{\abss{x} - \frac{1}{\lambda}, 0},
   \end{align*} 
   and the associated displacement function is
   \begin{align*}
      \phi(\lambda, x) = \abss{x - \prox{\nicefrac{f}{\lambda}}{x}} = \min \cbrac{\abss{x}, \frac{1}{\lambda}}.
   \end{align*}
   Thus, restricting to $\gS_\varepsilon = \cbrac{\varepsilon \leq \abss{x} \leq \abss{x_0}}$, we have  
   \begin{align*}
      m_f(\varepsilon, \lambda) = \min \cbrac{\varepsilon, \frac{1}{\lambda}}.
   \end{align*}
\end{example}

\begin{example}
   Let $f(x) = \frac{\mu}{2}x^2$ for some $\mu > 0$. Then 
   \begin{align*}
      \prox{\nicefrac{f}{\lambda}}{x}  = \frac{\lambda}{\lambda + \mu}x,
   \end{align*}
   and the displacement function is given by 
   \begin{align*}
      \abss{x - \prox{\nicefrac{f}{\lambda}}{x}} = \frac{\mu}{\lambda + \mu}\abss{x}.
   \end{align*}
   Therefore, 
   \begin{align*}
      m_f(\varepsilon, \lambda) = \min_{\varepsilon \leq \abss{x} \leq \abss{x_0}} \frac{\mu}{\lambda + \mu}\abss{x} = \frac{\mu}{\lambda + \mu}\varepsilon.
   \end{align*}
\end{example}

The preceding example extends naturally to higher-dimensional quadratic objectives.

\begin{lemma}[Simple Quadratics]
   \label{lemma:simple-quadatics}
   Consider the function $f(x) = \frac{1}{2}x^{\top}\mQ x$, where $\mQ \succeq \mO$ and $\mQ \neq \mO$.
   Then 
   \begin{align*}
      m_f(\varepsilon, \lambda) = \frac{\sigma_+}{\sigma_+ + \lambda}\varepsilon,
   \end{align*}
   where $\sigma_+$ is the smallest positive eigenvalue of $\mQ$.
\end{lemma}

\begin{remark}[Comparison with {\PPM}]
   For \TRPPM, fixing the regularization parameter $\lambda$ and selecting the trust-region radius according to \Cref{lemma:lnrate-fixed-lambda} yields the linear convergence bound
   \begin{align*}
      f(x_K) - f_{\inf} \leq \rbrac{\frac{1}{1 + \theta \cdot \nicefrac{m_f(\varepsilon,\lambda)}{d_0}}}^K
      \rbrac{f(x_0) - f_{\inf}}
   \end{align*}
   as long as the iterates remain outside the $\varepsilon$-neighborhood of the solution set.
   Here, this linear rate is achieved by stabilizing the iterates through the trust-region constraint, rather than by forcing the regularization parameter to vary.
\end{remark}

\subsection{Strongly Convex Case}
\label{sec:stronglyconvex}

The preceding discussion shows that, under the general convexity conditions of \Cref{assmp:pcc}, achieving linear convergence requires the introduction of a trust-region constraint.
However, it is well known that the classical {\PPM} already converges linearly when $f$ is strongly convex.
This naturally raises the question of whether our theory reflects the fact that a trust-region mechanism is unnecessary in the strongly convex setting.
As we show below, the answer is negative: when $f$ is strongly convex, the trust-region constraint becomes redundant, and linear convergence is guaranteed even in its absence.
\begin{theorem}
   \label{thm:combination-complicated}
   Let $f: \R^d \mapsto \R \cup \cbrac{+ \infty}$ satisfy \Cref{assmp:pcc} and let $\{x_k\}$ be the iterates of \Cref{eq:trppm-def} run with some sequences $\cbrac{\lambda_k}$ and $\cbrac{t_k}$. Then

    \begin{enumerate}
    \item For any chosen $t_k$ and $\lambda_k$, we have 
    \begin{align*}
        f(x_{k+1}) - f_{\inf} \leq q_k \rbrac{ f(x_k) - f_{\inf} },
    \end{align*}
    where $q_k = \rbrac{1 + \frac{\norm{x_k - x_{k+1}}}{\norm{x_{k+1} - x_\star^f}}}^{-1}$.

    \item If we further assume that $f$ is $\mu$-strongly convex, we have 
    \begin{align*}
        \gE_K \leq \prod_{k=0}^{K-1} \rbrac{1 + \min\cbrac{\frac{t_k}{d_0}, \frac{\mu}{\lambda_k}}}^{-1} \cdot \gE_0,
    \end{align*}
    where $\gE_K \eqdef f(x_K) - f_{\inf}$ and $K$ is the total number of iterations.
    \end{enumerate}
\end{theorem}

We now interpret the implications of \Cref{thm:combination-complicated} in both the convex and strongly convex regimes.

\textbullet~ In the convex case, if we guarantee that $t_k \leq \norm{x_k - \prox{\nicefrac{f}{\lambda_k}}{x_k}}$ in each iteration, then by the first brox theorem (\cref{fact:first-brox}), the next iterate is located on the boundary of the ball, which guarantees that $\norm{x_k - x_{k+1}} = t_k$.
Consequently, we obtain the inequality
\begin{align*}
   1 + \frac{\norm{x_k - x_{k+1}}}{\norm{x_{k+1} - x_\star^f}} \geq 1 + \frac{t_k}{d_0}.
\end{align*}
Linear convergence then follows from the fact that the sequence ${t_k}$ can be uniformly lower bounded on the set $\gS_\varepsilon$ when $\lambda_k \equiv \lambda$ is fixed.

\textbullet~ In the strongly convex case, we always have 
\begin{align*}
   \norm{x_k - x_{k+1}} = \min \cbrac{t_k, \norm{x_k - \prox{\nicefrac{f}{\lambda}}{x_k}}}.
\end{align*}
As in the proof of \Cref{thm:combination-complicated}, 
\begin{align*}
   1 + \frac{\norm{x_k - x_{k+1}}}{\norm{x_{k+1} - x_\star^f}} \geq \min \cbrac{1 + \frac{t_k}{d_0}, 1 + \frac{\mu}{\lambda_k}}.
\end{align*}
In particular, even if we set $t_k = +\infty$ (corresponding to the absence of a trust-region constraint), the quantity on the right-hand side remains uniformly lower bounded by $1 + \nicefrac{\mu}{\lambda_k}$.
This explains why, under strong convexity, the trust-region constraint is no longer necessary for achieving linear convergence.

\section{Conclusion}
\label{sec:conclusion}
In this work, we studied a trust-region stabilized proximal point framework and provided a unified analysis of its convergence behavior under minimal convexity assumptions.
By explicitly controlling the step size through the joint choice of the trust-region radius and the regularization parameter, we showed that linear descent can be obtained outside an arbitrary  prescribed neighborhood of the solution set without requiring strong convexity.
Our analysis clarified the distinct roles of trust-region constraints and quadratic regularization, established two complementary mechanisms for entering the linear regime, and revealed an exact equivalence between {\TRPPM} and {\BPM} in the active constraint setting (refer to \Cref{sec:mechanism}).


\section{Acknowledgements}
This work was supported by funding from King Abdullah University of Science and Technology (KAUST): i) KAUST Baseline Research Scheme, ii) Center of Excellence for Generative AI (award no. 5940). K.G. is supported by the Google PhD Fellowship.

\newpage




\bibliography{example_paper}
\bibliographystyle{icml2026}

\newpage
\appendix
\onecolumn

\section{Notation}
\label{sec:notation}
We work in $\R^d$ equipped with the standard inner product $\inner{\cdot}{\cdot}$ and norm $\norm{\cdot}$; for $x\in\R^d$ and $t\ge0$, $\gB_t(x)\eqdef\cbrac{z:\norm{z-x}\le t}$ denotes the closed Euclidean ball.
The objective $f:\R^d\to\R\cup\cbrac{+\infty}$ is assumed proper, closed and convex, with optimal value $f_{\inf}$ and solution set $\gX_\star^f\eqdef\argmin f$.
The subdifferential of $f$ at $x$ is denoted as $\partial f(x)$.
For a set $\gC$, $\dist{x}{\gC}$ denotes the Euclidean distance.
For any nonempty, closed, and convex set $\gC$, we define its indicator function $\iota_{\gC}: \R^d \to \R \cup \{+\infty\}$ by
\begin{align*}
    \iota_{\gC}(z) \eqdef
    \begin{cases}
        0 & \text{if } z\in \gC, \\
        +\infty & \text{otherwise}.
    \end{cases}
\end{align*}
This function is proper, closed, and convex.
The projection onto $\gC$ is $\ProjOn{\gC}{x}$.
For $\lambda>0$ the proximal operator is $\prox{\nicefrac{f}{\lambda}}{x}\eqdef\argmin_z\cbrac{f(z)+\frac{\lambda}{2}\norm{z-x}^2}$, with the convention $\prox{\nicefrac{f}{\lambda}}{x}=\ProjOn{\gX_\star^f}{x}$ when $\lambda=0$. 
The broximal operator with radius $t>0$ is $\brox{t}{f}{x}\eqdef\argmin_{z\in\gB_t(x)}f(z)$. 

Algorithm iterates are denoted by $\cbrac{x_k}$; {\PPM} uses $$x_{k+1}=\prox{\gamma_k f}{x_k} = \argmin_{z \in \R^d} \cbrac{f(z) + \frac{1}{2\gamma_k}\norm{z - x_k}^2},$$ while {\TRPPM} uses $$x_{k+1}=\argmin_{z\in\gB_{t_k}(x_k)}\cbrac{f(z)+\frac{\lambda_k}{2}\norm{z-x_k}^2}.$$ 
The displacement function is $\phi(\lambda,x)\eqdef\norm{x-\prox{\nicefrac{f}{\lambda}}{x}}$, and for $\varepsilon>0$, $m_f(\varepsilon,\lambda)$ denotes its minimum over points at distance at least $\varepsilon$ from $\gX_\star^f$ within the initial level set.

Finally, $\Exp{\cdot}$ denotes expectation, and $\xi_k$ denotes a (possibly stochastic) perturbation of the objective or model at iteration $k$.

\section{Facts and Lemmas}

\subsection{Facts}
\begin{fact}
   \label{fact:mono}
   If $f$ is proper and convex, then $\partial f$ is monotone.
\end{fact}

\begin{fact}[First Brox Theorem \citep{gruntkowska2025ball}]
   \label{fact:first-brox}
   Let $f:\R^d \mapsto \R \cup \cbrac{+\infty}$ satisfy \Cref{assmp:pcc} and choose $x \in \dom f$.
   Then
   \begin{enumerate}[label=(\roman*)]
        \item $\brox{t}{f}{x} \neq \emptyset$. 
   In addition, if $\gB_t(x) \cap \gX_f^\star \neq \emptyset$, then $\brox{t}{f}{x}$ is a non-empty subset of $\gX_f^\star$.
        \item If $\gB_t(x) \cap \gX_f^\star \neq \emptyset$, then $\brox{t}{f}{x}$ is a singleton lying on the boundary of $\gB_t(x)$.
    \end{enumerate}
\end{fact}

\begin{fact}[Second Brox Theorem \citep{gruntkowska2025ball}]
   \label{fact:second-brox}
   Let $f:\R^d \mapsto \R \cup \cbrac{+\infty}$ satisfy \Cref{assmp:pcc} and choose $x \in \dom f$ and $u \in \brox{t}{f}{x}$ for some $t > 0$.  
   Then, there exists $c_t(x) \geq 0$ such that 
   \begin{enumerate}[label=(\roman*)]
        \item $c_t(x)(x - u) \in \partial f(u)$.
        \item $f(y) - f(u) \geq c_t(x)\inner{x - u}{y - u}$ for all $y \in \R^d$.
   \end{enumerate}
\end{fact}

\begin{fact}[Three Point Identity]
   \label{fact:three-point-identity}
   For any $x, y, z \in \R^d$, we have 
   \begin{align*}
      \inner{x - z}{y - z} = \frac{1}{2}\norm{x - z}^2 + \frac{1}{2}\norm{z - y}^2 - \frac{1}{2}\norm{x - y}^2.
   \end{align*}
\end{fact}

\begin{fact}[Second Projection Theorem]
   \label{fact:second-proj}
   Let $\gC$ be a non-empty closed and convex set.
   Let $u \in \gC$.
   Then $u = \ProjOn{\gC}{x}$ if and only if 
   \begin{align*}
      \inner{x - u}{y - u} \leq 0, \quad \forall y \in \gC.
   \end{align*}
\end{fact}

\begin{fact}[Theorem 3.40 of \citet{beck2017first}]
   \label{fact:subdiff-calculus}
   Let $f_i : \R^d \to \R \cup \cbrac{+\infty}$, $i \in \cbrac{1,\dots,n}$, be proper convex functions such that $\cap_{i=1}^n \operatorname{ri}({\rm dom} f) \neq \emptyset$.
   Then, for any $x \in \R^d$, it holds that
   $$
      \partial\!\left( \sum_{i=1}^n f_i \right)(x) = \sum_{i=1}^n \partial f_i(x).
   $$
\end{fact}

\subsection{Useful Lemmas}

When analyzing {\TRPPM}, a natural first approach is to interpret it as an instance of {\BPM} applied to a noisy model $m_k$.
Under this viewpoint, one may directly invoke \Cref{fact:second-brox}, leading to conditions that closely resemble the standard {\BPM} analysis.
However, this approach yields suboptimal bounds.
To obtain sharper results, we therefore introduce the following refinement.
\begin{lemma}[Variant of Second Brox Theorem]
   \label{lemma:variant}
   Let $f$ be proper, closed and convex.
   Choose $x \in \operatorname{dom} f$ and let 
   \begin{align*}
      u \in \argmin_{z \in \gB_t(x)} \cbrac{f(z) + \frac{\lambda}{2}\norm{z - x}^2},
   \end{align*}
   where $\lambda \geq 0$ is a constant.
   Then, there exists some $c_t(x) \geq 0$ such that 
   \begin{itemize}
      \item[(i)] $\rbrac{c_t(x) + \lambda}(x - u) \in \partial f(u)$.
      \item[(ii)] $f(y) - f(u) \geq \rbrac{c_t(x) + \lambda}\inner{x - u}{y - u}$ for all $y \in \R^d$.
   \end{itemize}
\end{lemma}
\begin{remark}
   When $\lambda = 0$, \Cref{lemma:variant} reduces to the standard second brox theorem from \citet{gruntkowska2025ball}.
   When $\lambda > 0$, the coefficient $c_t(x) + \lambda$ is strictly positive.
\end{remark}

\subsection{Displacement as a Function of Regularizing Factor}
\label{sec:displacement-lambda}
In this section, we consider $\phi(\lambda, x) \eqdef \norm{x - \prox{\nicefrac{f}{\lambda}}{x}}$ as a function of $\lambda$.

\begin{lemma}[Range]
   \label{lemma:bounded-range}
   Let $f:\R^d \mapsto \R \cup \cbrac{+\infty}$ satisfy \Cref{assmp:pcc}.
   Then, for any $x\in\dom f$ and $\lambda \in [0, \infty]$
   \begin{align*}
      0 \leq \norm{x - \prox{\nicefrac{f}{\lambda}}{x}} \leq \norm{x - x_\star^f},
   \end{align*}
   where $x_\star^f \in \gX_\star^f$ is any minimizer of $f$.
   In particular, we have 
   \begin{align*}
   0 \leq \norm{x - \prox{\nicefrac{f}{\lambda}}{x}} \leq \dist{x}{\gX_\star^f}.
   \end{align*}
\end{lemma}

\begin{lemma}[Monotonicity]
   \label{lemma:monotonicity-displacement}
   Let $f:\R^d \mapsto \R \cup \cbrac{+\infty}$ satisfy \Cref{assmp:pcc}. 
   For any fixed $x \in \operatorname{dom} f$, $\phi(\lambda, x)$ is a monotone non-increasing function of $\lambda \geq 0$.
\end{lemma}
A natural question is whether a stronger conclusion, such as strict monotonicity, can be established. 
The following counterexample demonstrates that this cannot, in general, be achieved.

\begin{remark}[Strict Monotonicity]
   It is not possible to assert that, for an arbitrary function $f:\R^d \mapsto \R \cup \cbrac{+\infty}$ that satisfies \Cref{assmp:pcc} and any fixed point $x \notin \gX_\star^f$, the displacement function $\phi(\lambda, x)$ is strictly decreasing in $\lambda$.
   Indeed, consider the function $f(x) = \lvert x \rvert$.
   It is known that \citep[Lemma 6.5]{beck2017first} its proximal operator admits the closed-form expression
   \begin{align*}
      \prox{\nicefrac{f}{\lambda}}{x} = \sbrac{\lvert x \rvert - \frac{1}{\lambda}}_+ \cdot \sign(x).
   \end{align*}
   Fix $x = 0.5$.
   For $\lambda = 1$, we obtain
   \begin{align*}
      \prox{\nicefrac{f}{1}}{0.5} = 0.
   \end{align*}
   Similarly, for $\lambda = 2$, we still have
   \begin{align*}
      \prox{\nicefrac{f}{2}}{0.5} = 0.
   \end{align*}
   As a consequence, the corresponding displacement $\phi(\lambda, x)$ remains constant over this interval, demonstrating that strict monotonicity with respect to $\lambda$ cannot be guaranteed in general.
\end{remark}

\begin{lemma}[Limit]
   \label{lemma:limit}
   Let $f:\R^d \mapsto \R \cup \cbrac{+\infty}$ satisfy \Cref{assmp:pcc}.  
   For any $x \in \dom f$,
   \begin{align*}
      \lim \limits_{\lambda \rightarrow \infty} \phi(\lambda, x) = 0, \qquad \lim \limits_{\lambda \rightarrow 0^{+}} \phi(\lambda, x) = \dist{x}{\gX_f^\star}.
   \end{align*}
\end{lemma}

\begin{lemma}[Continuity I]
   \label{lemma:cont-distlambda}
   Let $f:\R^d \mapsto \R \cup \cbrac{+\infty}$ satisfy \Cref{assmp:pcc}.
   For any $x \in \dom f$, $\phi(\lambda, x)$ is a continuous function of $\lambda$ for $\lambda \geq 0$.
\end{lemma}

\subsection{Displacement as a Function of the Iterate}
In this section, we consider $\phi(\lambda, x)$ as a function of $x$.

\begin{lemma}
   \label{lemma:phi-PPM-sequence}
   Consider two consecutive iterations $k$ and $k+1$ of {\PPM}. If 
   \begin{align*}
      x_{k+1} &= \argmin_{z \in \R^d} \cbrac{f(z) + \frac{\lambda}{2}\norm{z - x_k}^2}; \qquad x_{k+2} = \argmin_{z \in \R^d} \cbrac{f(z) + \frac{\lambda}{2}\norm{z - x_{k+1}}^2},
   \end{align*}
   then $\norm{x_{k+2} - x_{k+1}} \leq \norm{x_k - x_{k+1}}$.
\end{lemma}

\begin{remark}
   We cannot obtain strict decrease in \Cref{lemma:phi-PPM-sequence} without further assumptions.
   Indeed, consider $f(x) = \mu \lvert x \rvert$, where $\mu > 0$. For a fixed $\lambda$, we know that 
   \begin{align*}
      \prox{\nicefrac{f}{\lambda}}{x} = \sign(x) \cdot \max\rbrac{\lvert x \rvert - \frac{\mu}{\lambda}, 0}.
   \end{align*}
   But then, for an initial point such that $x_0 > \frac{2\mu}{\lambda}$, we have 
   \begin{align*}
      x_1 = x_0 - \frac{\mu}{\lambda}, \quad x_2 = x_1 - \frac{\mu}{\lambda},
   \end{align*}
   and so $\norm{x_1 - x_0} = \norm{x_2 - x_1}$.
\end{remark}

\begin{lemma}[Lipschitz Continuity]
   \label{lemma:Lipschitz}
   Let $f: \R^d \mapsto \R \cup \cbrac{+\infty}$ satisfy \Cref{assmp:pcc}.
   Then for any fixed $\lambda > 0$, the mapping $\prox{\nicefrac{f}{\lambda}}{x}$ is $1$-Lipschitz, i.e.,
   \begin{align*}
      \norm{\prox{\nicefrac{f}{\lambda}}{x} - \prox{\nicefrac{f}{\lambda}}{y}} \leq \norm{x - y}, \quad \forall x, y \in \R^d.
   \end{align*}
\end{lemma}

\begin{lemma}[Continuity II]
   \label{lemma:cont-distx}
   Let $f: \R^d \mapsto \R \cup \cbrac{+\infty}$ satisfy \Cref{assmp:pcc}.
   Then the displacement function $\phi(\lambda, x)$ is a continuous function of $x$ for any $\lambda \geq 0$.
\end{lemma}

\subsection{Other Lemmas}
The following lemma is a generalized verison of \Cref{lemma:simplifiedlambda} for $q \geq 1$.
\begin{lemma}
   \label{lemma:ettet}
   Let $f:\R^d \mapsto \R \cup \cbrac{+\infty}$ satisfy \Cref{assmp:pcc} and \Cref{assmp:sharp-minima} with parameters $\alpha > 0$ and $q \geq 1$.
   Let $\{x_k\}$ be the iterates of \Cref{eq:trppm-def} run with a fixed trust-region radius $t_k \equiv t > 0$, and suppose that the current iterate lies outside the $t$-neighborhood of the solution set, i.e., $\dist{x_k}{\gX_\star^f} > t$.
   If the parameters $\lambda_k$ satisfy
   \begin{align*}
      0 < \lambda_k \leq \frac{2\alpha\rbrac{\dist{x}{\gX_\star^f} - t}^{q-1}}
     {\dist{x_k}{\gX_\star^f} + t},
   \end{align*}
   then it holds that
   \begin{align*}
      \norm{x_k - \prox{\nicefrac{f}{\lambda_k}}{x_k}} \geq t.
   \end{align*}
\end{lemma}
\begin{remark}
   When $q = 1$, we have 
   \begin{align*}
      0 < \lambda_k \leq \frac{2\alpha}{\dist{x_k}{\gX_\star^f} + t} \leq \frac{2\alpha}{\frac{1}{\alpha}\rbrac{f(x_k) - f_{\inf}} + t},
   \end{align*}
   which recovers the result in \Cref{lemma:simplifiedlambda}.
\end{remark}

\section{Relationship to {\BPM}}
\label{sec:mechanism}
We now explain why {\TRPPM} attains a linear convergence rate by relating it to {\BPM}.
\begin{lemma}
   \label{lemma:explaining}
   Let $f: \R^d \mapsto \R \cup \cbrac{+\infty}$ satisfy \Cref{assmp:pcc} and fix $x \in \dom f$, $t > 0$ and $\lambda > 0$. If $\prox{\nicefrac{f}{\lambda}}{x} \notin \gB_t(x)$ and $x_\star^f \notin \gB_t(x)$, then
   \begin{align*}
      \brox{t}{f}{x} = \brox{t}{f + \frac{\lambda}{2}\norm{\cdot - x}^2}{x}.
   \end{align*}
   In particular, for any non-tail iterate $x_k$ with $\lambda_k \leq \lambda_k^\star$ (where $\lambda_k^\star$ is defined in \Cref{lemma:existence}), we have
   \begin{align*}
      x^{\text{BPM}}_{k+1} = x^{\text{TRPPM}}_{k+1}.
   \end{align*}
\end{lemma}
This lemma shows that, for non-tail iterates, adding a sufficiently small quadratic regularization term to $f$ does not alter the next iterate of \BPM.
Since {\BPM} already converges linearly under convexity, this equivalence explains why {\TRPPM} inherits a stabilizing effect and achieves a linear convergence rate.

Moreover, unlike the original objective $f$ used in the subroutines of {\BPM}, the model $m_k = f + \frac{\lambda}{2}\norm{\cdot - x}^2$ is strongly convex, and is therefore typically easier to optimize and incurs lower computational overhead.

\section{Proofs of Lemmas and Theorems}

\subsection{\texorpdfstring{Proof of \Cref{lemma:existence}}{Proof of Lemma~\ref{lemma:existence}}}
\begin{proof}
   By \Cref{lemma:bounded-range}, we have $0 \leq \phi(\lambda, x_k) \leq \dist{x_k}{\gX_\star^f}$ for all $\lambda \geq 0$.
   Moreover, by \Cref{lemma:limit}, $\lim \limits_{\lambda \rightarrow 0^+} \phi(\lambda, x_k) = \dist{x_k}{\gX_\star^f} > t$.
   Hence, there exists $\lambda_1 > 0$ such that $\phi(\lambda_1, x_k) > t$.
   Similarly, since $\lim \limits_{\lambda \rightarrow \infty} \phi(\lambda, x_k) = 0$, there exists $\lambda_2 > 0$ such that $\phi(\lambda_2, x_k) < t$.
   By \Cref{lemma:monotonicity-displacement}, we know that $\lambda_1 < \lambda_2$.
   Since $\phi(\lambda, x_k)$ is continuous on the closed interval $[\lambda_1, \lambda_2]$ as a result of \Cref{lemma:cont-distlambda}, the intermediate value theorem ensures that there exists $\lambda_k^\star$ such that $\phi(\lambda_k^\star, x_k) = t$.
   We reach the conclusion by applying \Cref{lemma:monotonicity-displacement} again.
\end{proof}

\subsection{\texorpdfstring{Proof of \Cref{thm:linear-conv}}{Proof of Theorem~\ref{thm:linear-conv}}}

\begin{proof}
   Applying \Cref{lemma:variant}, with $y = x = x_k$, we have
   \begin{align}
      \label{eq:unknown-tmper1}
      f(x_{k+1}) - f(x_\star^f) &\leq f(x_k) - f(x_\star^f) - (c_{t_k}(x_k) + \lambda_k)\norm{x_{k+1} - x_k}^2.
   \end{align}
   On the other hand, applying \Cref{lemma:variant} with $x = x_k$ and $y = x_\star^f$, we have 
   \begin{align*}
      f(x_{k+1}) - f(x_\star^f) &\leq \rbrac{c_{t_k}(x_k) + \lambda_k}\inner{x_k - x_{k+1}}{x_{k+1} - x_\star^f} \leq  \rbrac{c_{t_k}(x_k) + \lambda_k}\norm{x_k - x_{k+1}}\norm{x_{k+1} - x_\star^f}.
   \end{align*}
   Assuming $x_{k+1} \neq x_\star^f$ (otherwise the minimizer has already been reached), we may rewrite this inequality as
   \begin{align*}
      \frac{f(x_{k+1}) - f(x_\star^f)}{\norm{x_{k+1} - x_\star^f}} \cdot \norm{x_k - x_{k+1}} \leq \rbrac{c_{t_k}(x_k) + \lambda_k}\norm{x_k - x_{k+1}}^2.
   \end{align*}
   Substituting this bound into \cref{eq:unknown-tmper1} gives
   \begin{align}
      \label{eq:ref-org}
      f(x_{k+1}) - f(x_\star^f) \leq f(x_k) - f(x_\star^f) - \frac{\norm{x_k - x_{k+1}}}{\norm{x_{k+1} - x_\star^f}} \cdot \rbrac{f(x_{k+1}) - f(x_\star^f)}.
   \end{align}
   At each iteration, we choose $\lambda_k \in (0,\lambda_k^\star]$.
   By \Cref{lemma:existence}, as long as $x_k$ lies outside the $t$-neighborhood of $\gX_\star^f$, we have $\norm{x_k - \prox{\nicefrac{f}{\lambda_k}}{x_k}} \geq t$.
   Consequently, the ball constraint is active at every such iteration, and hence $\norm{x_k - x_{k+1}} = t$.
   Substituting this into the inequality above, and noting that $\cbrac{\norm{x_{k+1} - x_\star^f}}$ is a nonincreasing sequence, we obtain
   \begin{align*}
      \rbrac{1 + \frac{t}{d_0}}\rbrac{f(x_{k+1}) - f_{\inf}} \leq f(x_{k}) - f_{\inf}.
   \end{align*}
   Unrolling the inequality for $k=0, \hdots, K-1$, where $x_K$ is the last iterate outside of the $t$-neighborhood, we get 
   \begin{align*}
      f(x_{K}) - f_{\inf} \leq \rbrac{\frac{1}{1 + \nicefrac{t}{d_0}}}^K\rbrac{f(x_0) - f_{\inf}}.
   \end{align*}
\end{proof}

\subsection{\texorpdfstring{Proof of \Cref{lemma:roughestimate}}{Proof of Lemma~\ref{lemma:roughestimate}}}

\begin{proof}
   The case $\lambda_k^\star = 0$ is trivial.
   Assume henceforth that $\lambda_k^\star > 0$.
   From the optimality condition, we have
   \begin{align*}
      f(\prox{\nicefrac{f}{\lambda_k^\star}}{x_k}) + \frac{\lambda_k^\star}{2}\norm{x_k - \prox{\nicefrac{f}{\lambda_k^\star}}{x_k}}^2 \leq f(x_k).
   \end{align*}
   Thus, 
   \begin{align*}
      \norm{x_k - \prox{\nicefrac{f}{\lambda_k^\star}}{x_k}}^2 \leq \frac{2}{\lambda_k^\star}\rbrac{f(x_k) - f(\prox{\nicefrac{f}{\lambda_k^\star}}{x_k})} \leq \frac{2}{\lambda_k^\star}\rbrac{f(x_k) - f_{\inf}}.
   \end{align*}
   The claim follows by combining the above inequality with the fact that $\norm{x_k - \prox{\nicefrac{f}{\lambda_k^\star}}{x_k}} \geq t$.
\end{proof}

\subsection{\texorpdfstring{Proof of \Cref{lemma:simplifiedlambda}}{Proof of Lemma~\ref{lemma:simplifiedlambda}}}
\begin{proof}
   We prove the contrapositive of the statement.
   Assume that $\norm{x_k - \prox{\nicefrac{f}{\lambda_k}}{x_k}} < t$.
   By the optimality condition of the proximal subproblem, for any $x_\star^f \in \gX_\star^f$,
   \begin{align*}
      f(\prox{\nicefrac{f}{\lambda_k}}{x_k}) + \frac{\lambda_k}{2}\norm{\prox{\nicefrac{f}{\lambda_k}}{x_k} - x_k}^2 \leq f(x_\star^f) + \frac{\lambda_k}{2}\norm{x_k - x_\star^f}^2.
   \end{align*}
   Rearranging terms,
   \begin{align*}
      f(\prox{\nicefrac{f}{\lambda_k}}{x_k}) - f(x_\star^f) \leq \frac{\lambda_k}{2}\rbrac{\norm{x_k - x_\star^f}^2  - \norm{\prox{\nicefrac{f}{\lambda_k}}{x_k} - x_k}^2}
   \end{align*}
   The above inequality holds for arbitrary $x_\star^f$, we obtain
   \begin{align}
      \label{eq:keyalpha1}
      f(\prox{\nicefrac{f}{\lambda_k}}{x_k})  - f_{\inf} \leq \frac{\lambda_k}{2}\rbrac{\distsq{x_k}{\gX_\star^f} - \norm{\prox{\nicefrac{f}{\lambda_k}}{x_k} - x_k}^2}.
   \end{align}
   On the other hand, invoking \Cref{assmp:sharp-minima}, we have 
   \begin{align*}
      f(\prox{\nicefrac{f}{\lambda_k}}{x_k}) - f_{\inf} \geq \alpha \cdot \dist{\prox{\nicefrac{f}{\lambda_k}}{x_k}}{\gX_\star^f}.
   \end{align*}
   Applying the triangle inequality gives
   \begin{align}
      \label{eq:keyalpha2}
      f(\prox{\nicefrac{f}{\lambda_k}}{x_k}) - f_{\inf} &\geq \alpha\rbrac{\norm{x_k - \ProjOn{\gX_\star^f}{\prox{\nicefrac{f}{\lambda_k}}{x_k}}} - \norm{x_k - \prox{\nicefrac{f}{\lambda_k}}{x_k}}} \notag \\
      &\geq \alpha \rbrac{\dist{x_k}{\gX_\star^f} - \norm{x_k - \prox{\nicefrac{f}{\lambda_k}}{x_k}}}.
   \end{align}
   Combining \cref{eq:keyalpha1} and \cref{eq:keyalpha2}, we have 
   \begin{align*}
      & \dist{x_k}{\gX_\star^f} - \norm{x_k - \prox{\nicefrac{f}{\lambda_k}}{x_k}} \\
      & \leq \frac{\lambda_k}{2\alpha}\rbrac{\distsq{x_k}{\gX_\star^f} - \norm{\prox{\nicefrac{f}{\lambda_k}}{x_k} - x_k}^2} \\
      & = \frac{\lambda_k}{2\alpha}\rbrac{\dist{x_k}{\gX_\star^f} - \norm{\prox{\nicefrac{f}{\lambda_k}}{x_k} - x_k}}\rbrac{\dist{x_k}{\gX_\star^f} + \norm{\prox{\nicefrac{f}{\lambda_k}}{x_k} - x_k}}.
   \end{align*}
   Since we assume $\dist{x_k}{\gX_\star^f} \geq t$ and $\norm{x_k - \prox{\nicefrac{f}{\lambda_k}}{x_k}} < t$, we can divide both sides of the inequality by $\dist{x_k}{\gX_\star^f} - \norm{x_k - \prox{\nicefrac{f}{\lambda_k}}{x_k}} > 0$ to obtain
   \begin{align*}
      \norm{\prox{\nicefrac{f}{\lambda_k}}{x_k} - x_k} + \dist{x_k}{\gX_\star^f} \geq \frac{2\alpha}{\lambda_k},
   \end{align*}
   or equivalently, 
   \begin{align*}
      \norm{\prox{\nicefrac{f}{\lambda_k}}{x_k} - x_k} \geq \frac{2\alpha}{\lambda_k} - \dist{x_k}{\gX_\star^f}.
   \end{align*}
   Now, by \Cref{assmp:sharp-minima}, 
   \begin{align*}
      \dist{x_k}{\gX_\star^f} \leq \frac{1}{\alpha} \rbrac{f(x_k) - f_{\inf}}.
   \end{align*}
   Combining the above two inequalities yields
   \begin{align*}
      \norm{\prox{\nicefrac{f}{\lambda_k}}{x_k} - x_k} \geq \frac{2\alpha}{\lambda_k} - \frac{1}{\alpha}\rbrac{f(x_k) - f_{\inf}}.
   \end{align*}
   Rearranging 
   \begin{align*}
      \frac{2\alpha}{\lambda_k} - \frac{1}{\alpha}\rbrac{f(x_k) - f_{\inf}} < t,
   \end{align*}
   gives
   \begin{align*}
      \lambda_k > \frac{2\alpha^2}{f(x_k) - f_{\inf} + \alpha t},
   \end{align*}
   which concludes the proof.
\end{proof}

\subsection{\texorpdfstring{Proof of \Cref{col:thm:linear-conv}}{Proof of Corollary~\ref{col:thm:linear-conv}}}
\begin{proof}
   Since we choose
   \begin{align*}
      0 \leq \lambda_k \leq \frac{2\alpha^2}{f(x_k) - f_{\inf} + \alpha t},
   \end{align*}
   it follows from \Cref{lemma:simplifiedlambda} that $\norm{x_k - \prox{\nicefrac{f}{\lambda_k}}{x_k}} > t$.
   The remainder of the argument then proceeds identically to the proof of \Cref{thm:linear-conv}.
\end{proof}

\subsection{\texorpdfstring{Proof of \Cref{lemma:uniform-lowerbound}}{Proof of Lemma~\ref{lemma:uniform-lowerbound}}}
\begin{proof}

   We first note that the set $\cbrac{x: \norm{x - x_\star^f} \leq \norm{x_0 - x_\star^f}}$ is closed and bounded in $\R^d$, and therefore compact.
   Moreover, the distance function $\dist{x}{\gX_\star^f}$ is continuous, which implies that the set $\cbrac{x: \dist{x}{\gX_f^\star} \geq \varepsilon}$ is closed.
   By definition, the set $\gS_\varepsilon$ is the intersection of these two sets, and is thus closed and bounded. Hence, $\gS_\varepsilon$ is compact.
   We know $\phi(\lambda, x)$ is a continuous function of $x$ for fixed $\lambda$ according to \Cref{lemma:cont-distx}.
   By the Weierstrass extreme value theorem, $\phi(\lambda, x)$ attains a minimum on the compact set $\gS_\varepsilon$, which we denote as 
   \begin{align*}
      m_f(\varepsilon, \lambda) = \min_{x \in \gS_\varepsilon} \norm{x - \prox{\nicefrac{f}{\lambda}}{x}}.
   \end{align*}
   Since $\gS_\varepsilon \cap \gX_f^\star = \emptyset$, we know that $m_f(\varepsilon, \lambda) > 0$, as otherwise there would exist $x_\star^f \in \gS_\varepsilon$, contradicting the definition of $\gS_\varepsilon$.
   Thus, for all $x \in \gS_\varepsilon$, we have 
   \begin{align*}
      \norm{x - \prox{\nicefrac{f}{\lambda}}{x}} \geq m_f(\varepsilon, \lambda).
   \end{align*}
\end{proof}

\subsection{\texorpdfstring{Proof of \Cref{lemma:lnrate-fixed-lambda}}{Proof of Theorem~\ref{lemma:lnrate-fixed-lambda}}}
\begin{proof}
   When $\lambda = 0$, the method reduces to {\BPM}, and the claim holds trivially by choosing $t \equiv m_f(\varepsilon,0) = \varepsilon$.

   We now turn to the case $\lambda > 0$.
   Applying \Cref{lemma:variant} with \Cref{fact:three-point-identity}, we obtain
   \begin{align}
     \label{eq:dnknw-1}
     f(x_{k+1}) - f(x_\star^f) \leq \frac{c_{t_k}(x_k) + \lambda}{2} \cdot \rbrac{\norm{x_k - x_\star^f}^2 - \norm{x_{k+1} - x_\star^f}^2 - \norm{x_k - x_{k+1}}^2}.
   \end{align}
   Thus for any $x_\star^f$ 
   \begin{align*}
      \norm{x_{k+1} - x_\star^f} \leq \norm{x_k - x_\star^f}, \quad \text{ for all } k.
   \end{align*}
   Consequently, for all iterates, we have $\norm{x_k - x_\star^f} \leq \norm{x_0 - x_\star^f}$. 
   We now focus on the non-tail iterates, namely those satisfying $\dist{x_k}{\gX_\star^f} \geq \varepsilon$.
   By the above inequality, all such iterates belong to the set $\gS_\varepsilon$.
   Applying \Cref{lemma:uniform-lowerbound}, we obtain that for every non-tail iterate $x_k$,
   \begin{align*}
      \norm{x_k - \prox{\nicefrac{f}{\lambda}}{x_k}} \geq m_f(\varepsilon,\lambda) > 0.
   \end{align*}
   Since this lower bound is uniform, for any fixed $\theta \in (0,1]$ we may set $t_k \equiv t = \theta \cdot m_f(\varepsilon,\lambda)$.
   Using this choice in the descent inequality \cref{eq:ref-org} established earlier, we obtain 
   \begin{align*}
      \rbrac{1 + \frac{\theta \cdot m_f(\varepsilon, \lambda)}{d_0}}\rbrac{f(x_{k+1}) - f(x_\star^f)} \leq f(x_k) - f(x_\star^f),
   \end{align*}
   and hence
   \begin{align*}
      f(x_{k+1}) - f_{\inf} \leq \rbrac{\frac{1}{1 + \nicefrac{\theta \cdot m_f(\varepsilon, \lambda)}{d_0}}} \rbrac{f(x_k) - f_{\inf}}.
   \end{align*}
   This establishes linear convergence for all non-tail iterates $x_k$ satisfying $\dist{x_k}{\gX_\star^f} \geq \varepsilon$.
\end{proof}

\subsection{\texorpdfstring{Proof of \Cref{lemma:undecided}}{Proof of Lemma~\ref{lemma:undecided}}}

\begin{proof}
   Let $0 \leq \lambda_1 < \lambda_2$. For any fixed $x$, $\phi(\lambda, x) = \norm{x - \prox{\nicefrac{f}{\lambda}}{x}}$ is a non-increasing function of $\lambda \geq 0$, according to \Cref{lemma:monotonicity-displacement}.
   Consequently, $\norm{x - \prox{\nicefrac{f}{\lambda_2}}{x}} \leq \norm{x - \prox{\nicefrac{f}{\lambda_1}}{x}}$ for any $x \in \gS_\varepsilon$.
   Taking infimum on both sides gives 
   \begin{align*}
      \inf_{x \in \gS_\varepsilon} \norm{x - \prox{\nicefrac{f}{\lambda_2}}{x}} \leq \inf_{x \in \gS_\varepsilon} \norm{x - \prox{\nicefrac{f}{\lambda_1}}{x}},
   \end{align*}
   which is exactly $m_f(\varepsilon, \lambda_2) \leq m_f(\varepsilon, \lambda_1)$.
\end{proof}

\subsection{\texorpdfstring{Proof of \Cref{lemma:cont-m}}{Proof of Lemma~\ref{lemma:cont-m}}}
\begin{proof}
   Consider any sequence $\lambda_n\to \lambda$ with $\lambda_n \geq 0$.
   \paragraph{Lower bound.}
   Let $\cbrac{n_k}$ be a subsequence such that $\lim \limits_{k\to\infty} m_f(\varepsilon,\lambda_{n_k}) = \liminf \limits_{n\to\infty} m_f(\varepsilon,\lambda_n)$.
   For each $k$, choose $x_{n_k}\in \gS_\varepsilon$ such that $m_f(\varepsilon, \lambda_{n_k}) \;=\; \phi(\lambda_{n_k}, x_{n_k})$.
   Since $\gS_\varepsilon$ is compact, there exists a further subsequence (still indexed by $n_k$) such that $x_{n_k}\to \bar x \in \gS_\varepsilon$.
   We next prove that $\phi(\lambda_{n_k}, x_{n_k})\to \phi(\lambda, \bar x)$,
   which implies
   \begin{align*}
      \liminf_{n \to \infty} m_f(\varepsilon,\lambda_n)
      = \lim_{k\to\infty} m_f(\varepsilon,\lambda_{n_k})
      = \lim_{k\to\infty}\phi(\lambda_{n_k}, x_{n_k})
      = \phi(\lambda, \bar x)
      \geq \min_{x\in\gS_\varepsilon}\phi(\lambda, x)
      = m_f(\varepsilon, \lambda).
   \end{align*}

   \paragraph{Continuity of $\phi(\lambda, x)$ along $(\lambda_{n_k}, x_{n_k})$.}
   We consider two cases.

   \begin{itemize}
        \item $\lambda > 0$:
        For all sufficiently large $k$, we have $\lambda_{n_k} > 0$.
        For the sake of simplicity, let us denote $z_{n_k} \eqdef \prox{\nicefrac{f}{\lambda_{n_k}}}{x_{n_k}}$.
        Fix the minimizer $x_\star^f\in\gX_\star^f$.
        By the optimality of $z_{n_k}$, we have
        \begin{align*}
        f(z_{n_k}) + \frac{\lambda_{n_k}}{2}\norm{z_{n_k} - x_{n_k}}^2
        \leq f(x_\star^f) + \frac{\lambda_{n_k}}{2}\norm{x_\star^f - x_{n_k}}^2.
        \end{align*}
        Since $f(z_{n_k}) \geq f(x_\star^f)$, we obtain $\norm{z_{n_k} - x_{n_k}} \leq \norm{x_\star^f - x_{n_k}}$.
        Recall that $x_{n_k}\in \gS_\varepsilon$ and $\gS_\varepsilon$ is bounded. Since the right hand side is bounded, $\cbrac{z_{n_k}}$ is bounded as well.
        Therefore, along a further subsequence (still indexed by $n_k$) we may assume $z_{n_k}\to \bar z$ for some $\bar z\in\R^d$.
        Now, take any $z \in \R^d$. By optimality of $z_{n_k}$,
        \begin{align*}
        f(z_{n_k}) + \frac{\lambda_{n_k}}{2}\norm{z_{n_k}-x_{n_k}}^2
        \le
        f(z) + \frac{\lambda_{n_k}}{2}\norm{z-x_{n_k}}^2.
        \end{align*}
        Taking $\lim\inf$ on the left and limit on the right, using the assumption that $f$ is closed, and that
        $\lambda_{n_k} \to \lambda$, $x_{n_k} \to \bar x$, $z_{n_k} \to \bar z$, we have
        \begin{align*}
        f(\bar z) + \frac{\lambda}{2}\norm{\bar z - \bar x}^2
        \le
        f(z) + \frac{\lambda}{2}\norm{z - \bar x}^2
        \qquad \forall z \in \R^d.
        \end{align*}
        Thus $\bar z$ minimizes $z\mapsto f(z)+\frac{\lambda}{2}\norm{z-\bar x}^2$, i.e. $\bar z = \prox{\nicefrac{f}{\lambda}}{\bar x}$.
        Since the prox minimizer is unique for $\lambda>0$, every convergent subsequence of $\cbrac{z_{n_k}}$ must converge to $\prox{\nicefrac{f}{\lambda}}{\bar x}$, hence the whole sequence satisfies $z_{n_k} \to \prox{\nicefrac{f}{\lambda}}{\bar x}$.
        As a result,
        \begin{align*}
        \phi(\lambda_{n_k}, x_{n_k}) = \norm{x_{n_k} - z_{n_k}} \to \norm{\bar x-\prox{\nicefrac{f}{\lambda}}{\bar x}} = \phi(\lambda, \bar x).
        \end{align*}
        
        \item $\lambda = 0$:
        If $\lambda_{n_k}=0$ for infinitely many $k$, then along that subsequence $\phi(\lambda_{n_k}, x_{n_k})= \dist{x_{n_k}}{\gX_\star^f}\to \dist{\bar x}{\gX_\star^f}=\phi(0, \bar x)$ by continuity of the distance function.
        Otherwise, we may assume $\lambda_{n_k} > 0$ for all sufficiently large $k$.
        For any minimizer $x_\star^f \in \gX_\star^f$, by optimality of $z_{n_k}$ we have
        \begin{align*}
        f(z_{n_k}) + \frac{\lambda_{n_k}}{2}\norm{z_{n_k}-x_{n_k}}^2 \leq f(x_\star^f) + \frac{\lambda_{n_k}}{2}\norm{x_\star^f - x_{n_k}}^2.
        \end{align*}
        Since $f(z_{n_k})\ge f(x_\star^f)$, this implies $\norm{z_{n_k}-x_{n_k}} \leq \norm{x_\star^f - x_{n_k}}$.
        Taking the infimum over $x_\star^f \in \gX_\star^f$ yields
        \begin{align}
        \label{eq:star}
        \norm{x_{n_k} - z_{n_k}} \leq \dist{x_{n_k}}{\gX_\star^f}.
        \end{align}
        Hence
        \begin{align*}
        \limsup_{k \to \infty} \norm{x_{n_k} - z_{n_k}}
        \leq \lim_{k\to\infty}\dist{x_{n_k}}{\gX_\star^f}
        = \dist{\bar x}{\gX_\star^f},
        \end{align*}
        where we used continuity of $x \mapsto \dist{x}{\gX_\star^f}$.
        Conversely, note that \cref{eq:star} and boundedness of $\cbrac{x_{n_k}}$ imply that $\cbrac{z_{n_k}}$ is bounded.
        Choose a subsequence (still indexed by $n_k$) such that $\norm{x_{n_k}-z_{n_k}}$ converges to its limit inferior and $z_{n_k}\to \bar z$.
        Comparing objective values at $z_{n_k}$ and $x_\star^f \in \gX_\star^f$ gives
        \begin{align*}
        f(z_{n_k})\le f(x_\star^f) + \frac{\lambda_{n_k}}{2}\norm{x_\star^f-x_{n_k}}^2.
        \end{align*}
        Letting $k \to \infty$ and using $\lambda_{n_k} \to  0$ and boundedness of $\cbrac{x_{n_k}}$ gives $\limsup f(z_{n_k}) \leq f(x_\star^f)$.
        By lower semicontinuity of~$f$,
        \begin{align*}
        f(\bar z) \leq \liminf_{k \to \infty} f(z_{n_k}) \leq f(x_\star^f),
        \end{align*}
        so $\bar z \in \gX_\star^f$.
        Therefore,
        \begin{align*}
        \liminf_{k \to \infty} \norm{x_{n_k} - z_{n_k}} = \lim_{k\to\infty}\norm{x_{n_k} - z_{n_k}} = \norm{\bar x - \bar z} \geq \dist{\bar x}{\gX_\star^f}.
        \end{align*}
        Combining limit superior and limit inferior gives $\norm{x_{n_k} - z_{n_k}} \to \dist{\bar x}{\gX_\star^f} = \phi(0, \bar x) = \phi(\lambda, \bar x)$.
   \end{itemize}

   Thus, in both cases we have shown $\phi(\lambda_{n_k}, x_{n_k})\to \phi(\lambda, \bar x)$.
        
    \paragraph{Upper bound.}
    It remains to prove the matching upper bound for the original sequence $\cbrac{\lambda_n}$.
    Let $\hat x \in \gS_\varepsilon$ be any minimizer of $\phi(\lambda, \cdot)$ over $\gS_\varepsilon$. Then $
    m_f(\varepsilon, \lambda_n) = \min_{x\in\gS_\varepsilon}\phi(\lambda_n, x) \leq \phi(\lambda_n, \hat x)$.
    Applying the above reasoning to the pair $(\lambda_n,\hat x)$ (i.e., with $x_n \equiv \hat x$), we have $\phi(\lambda_n, \hat x) \to \phi(\lambda, \hat x) = m_f(\varepsilon, \lambda)$, and thus
    \begin{align*}
    \limsup_{n \to \infty} m_f(\varepsilon,\lambda_n)\leq m_f(\varepsilon, \lambda).
    \end{align*}
    Combining the $\liminf$ and $\limsup$ bounds yields $m_f(\varepsilon, \lambda_n) \to m_f(\varepsilon, \lambda)$.
    Since the sequence $\cbrac{\lambda_n}$ was arbitrary, $\lambda \mapsto m_f(\varepsilon, \lambda)$ is continuous on $[0, \infty)$.

\end{proof}

\subsection{\texorpdfstring{Proof of \Cref{lemma:simple-quadatics}}{Proof of Lemma~\ref{lemma:simple-quadatics}}}

\begin{proof}
   For $\lambda > 0$, we have 
   \begin{align*}
      \prox{\nicefrac{f}{\lambda}}{x} = \rbrac{\lambda \mI + \mQ}^{-1}\lambda x, \qquad x - \prox{\nicefrac{f}{\lambda}}{x} = \mQ\rbrac{\lambda \mI + \mQ}^{-1} x.
   \end{align*}
   Consider the spectral decomposition $\mQ = \mU \operatorname{diag}\rbrac{\sigma_1, \hdots, \sigma_d} \mU^{\top}$, where $\mU$ is an orthogonal matrix and $\sigma_1, \hdots, \sigma_d$ are the eigenvalues of $\mQ$.
   Then
   \begin{align*}
      \mQ(\lambda \mI + \mQ)^{-1} = \mU \operatorname{diag}\rbrac{\frac{\sigma_1}{\sigma_1 + \lambda}, \hdots, \frac{\sigma_d}{\sigma_d + \lambda}} \mU^{\top}.
   \end{align*}
   Let $y = \mU^{\top} x$. Since $\mU$ is orthogonal, we have $\norm{\mU x} = \norm{y}$.
   Now consider two cases:
   \begin{itemize}
        \item If $\mQ \succ \mO$, the set of minimizer is given by $\gX_\star^f = \cbrac{0}$, and we have 
        \begin{align*}
        \phi(\lambda, x) = \norm{\mQ\rbrac{\lambda \mI + \mQ}^{-1} x} &= \norm{\mU \operatorname{diag}\rbrac{\frac{\sigma_1}{\sigma_1 + \lambda}, \hdots, \frac{\sigma_d}{\sigma_d + \lambda}} y} \\
        &= \norm{\operatorname{diag}\rbrac{\frac{\sigma_i}{\sigma_i + \lambda}}y} = \sqrt{\sum_{i=1}^{d}\rbrac{\frac{\sigma_i}{\sigma_i + \lambda}}^2y_i^2} \geq \frac{\sigma_{\min}}{\sigma_{\min} + \lambda}\norm{x}.
        \end{align*}
        Thus, 
        \begin{align*}
        m_f(\varepsilon, \lambda) = \min_{x\in \gS_\varepsilon} \phi(\lambda, x) \geq \frac{\sigma_{\min}}{\sigma_{\min} + \lambda}\varepsilon = \frac{\mu}{\mu + \lambda}\varepsilon.
        \end{align*}
        This bound is tight, as equality is achieved when $\norm{x}=\varepsilon$ and $x$ is an eigenvector of $\mQ$ associated with $\sigma_{\min}$.
        \item If $\mQ \succeq \mO$ with some $\sigma_i = 0$. In this case, the set of minimizers of $f$ is $\gX_f^\star = \operatorname{ker} \mQ$. For $x \in \operatorname{ker} \mQ$, $\mU\operatorname{diag}(\sigma_i)y = \mQ x = 0$ is equivalent to $\operatorname{diag}(\sigma_i)y = 0$, which means that the set of $y$ is exactly 
        \begin{align*}
        \gY = \cbrac{y: y_i = 0 \text{ for all indices with } \sigma_i > 0}.
        \end{align*}
        Notice that 
        \begin{align*}
        \dist{x}{\operatorname{ker} \mQ} = \dist{y}{\gY} = \sum_{\cbrac{i:\sigma_i > 0}} y_i^2, \qquad \phi(\lambda, x) = \sqrt{\sum_{\cbrac{i:\sigma_i > 0}} \rbrac{\frac{\sigma_i}{\sigma_i + \lambda}}^2y_i^2}
        \end{align*}
        Let $\sigma_+$ be the smallest positive eigenvalue of $\mQ$. For other indices, we have 
        \begin{align*}
        \frac{\sigma_i}{\sigma_i + \lambda} \geq \frac{\sigma_+}{\sigma_+ + \lambda}.
        \end{align*}
        Thus, 
        \begin{align*}
        \phi(\lambda, x) \geq \frac{\sigma_+}{\sigma_+ + \lambda}\dist(x, \gX_\star^f).
        \end{align*}
        Since $\gS_\varepsilon$ constrains $\dist{x}{\gX_\star^f} \geq \varepsilon$, we get 
        \begin{align*}
        m_f(\varepsilon, \lambda) \geq \frac{\sigma_+}{\sigma_+ + \lambda}\varepsilon.
        \end{align*}
        We realize the bound is tight once we choose $\norm{x} = \varepsilon$, which corresponds to the eigenvector of $\sigma_+$.
   \end{itemize}
\end{proof}

\subsection{\texorpdfstring{Proof of \Cref{thm:combination-complicated}}{Proof of Theorem~\ref{thm:combination-complicated}}}

\begin{proof}
   Following the same argument as in \Cref{thm:linear-conv} (refer to \cref{eq:ref-org}), we obtain
   \begin{align}\label{eq:oabiwsvdz}
      f(x_{k+1}) - f_{\inf} \leq f(x_k) - f_{\inf} - \frac{\norm{x_k - x_{k+1}}}{\norm{x_{k+1} - x_\star^f}} \cdot \rbrac{f(x_{k+1}) - f_{\inf}}.
   \end{align}
   Rearranging terms yields
   \begin{align}
      \label{eq:org-per-step-descent}
      \rbrac{1 + \frac{\norm{x_k - x_{k+1}}}{\norm{x_{k+1} - x_\star^f}}} \cdot \rbrac{f(x_{k+1}) - f(x_\star^f)} \leq f(x_k) - f(x_\star^f).
   \end{align}
   Up to this point, the descent inequality holds for arbitrary choices of $t_k$. 
   We now recall the structure of the update.
   In the standard {\BPM} setting, unless the minimizer is found in one step, the iterate $x_{k+1}$ lies on the boundary of the trust-region.
   More precisely,
   \begin{align*}
      \norm{x_k - x_{k+1}} = \min \cbrac{t_k, \norm{x_k - x_\star^f}},
   \end{align*}
   and hence $\norm{x_k - x_{k+1}} = t_k$ whenever $t_k \le \norm{x_k - x_\star^f}$.
   With an additional quadratic regularizer, the same reasoning applies, yielding
   \begin{align*}
      \norm{x_k - x_{k+1}} = \min \cbrac{t_k, \norm{x_k - \prox{\nicefrac{f}{\lambda_k}}{x_k}}},
   \end{align*}
   since $\prox{\nicefrac{f}{\lambda_k}}{x_k}$ is the global minimizer of the model. 
   We therefore distinguish two cases.
   \begin{itemize}
        \item \textbf{Active constraint.} If $t_k \leq \norm{x_k - \prox{\nicefrac{f}{\lambda_k}}{x_k}}$, then the ball constraint is active and $\norm{x_k - x_{k+1}} = t_k$.
        By \cref{eq:dnknw-1}, the sequence $\cbrac{\norm{x_k - x_\star^f}}$ is non-increasing, and hence $\norm{x_{k+1} - x_\star^f} \leq \norm{x_0 - x_\star^f} \eqdef d_0$.
        Consequently,
        \begin{align}
        \label{eq:lower-bound-eq1}
        1 + \frac{\norm{x_k - x_{k+1}}}{\norm{x_{k+1} - x_\star^f}} \geq 1 + \frac{t_k}{d_0}.
        \end{align}
        \item \textbf{Inactive constraint.} If $t_k > \norm{x_k - \prox{\nicefrac{f}{\lambda_k}}{x_k}}$, then $x_{k+1} = \prox{\nicefrac{f}{\lambda_k}}{x_k}$.
        In this case, to obtain a meaningful lower bound we assume that $f$ is $\mu$-strongly convex.
        Let $\tilde{\nabla}f(\prox{\nicefrac{f}{\lambda_k}}{x_k}) = \lambda_k(x - x_{k+1}) \in \partial f(\prox{\nicefrac{f}{\lambda_k}}{x_k})$ and note that
        $\tilde{\nabla}f(x_\star^f) \eqdef 0 \in \partial f(x_\star^f)$.
        Then
        \begin{align*}
        \norm{x_k - \prox{\nicefrac{f}{\lambda_k}}{x_k}} = \frac{1}{\lambda_k} \cdot \norm{\tilde{\nabla} f(\prox{\nicefrac{f}{\lambda_k}}{x_k})} &= \frac{1}{\lambda_k} \cdot \norm{\tilde{\nabla} f(\prox{\nicefrac{f}{\lambda_k}}{x_k}) - \tilde{\nabla} f(x_\star^f)} \\
        &\geq \frac{\mu}{\lambda_k} \cdot \norm{\prox{\nicefrac{f}{\lambda_k}}{x_k} - x_\star^f},
        \end{align*}
        where the last inequality follows from strong convexity.
        Therefore,  
        \begin{align}
        \label{eq:lower-bound-eq2}
        1 + \frac{\norm{x_k - x_{k+1}}}{\norm{x_{k+1} - x_\star^f}} \geq 1 + \frac{\mu}{\lambda_k} \cdot \frac{\norm{\prox{\nicefrac{f}{\lambda_k}}{x_k} - x_\star^f}}{\norm{x_{k+1} - x_\star^f}} = 1 + \frac{\mu}{\lambda_k}.
        \end{align}
   \end{itemize}

   Combining \cref{eq:lower-bound-eq1} and \cref{eq:lower-bound-eq2}, we obtain
   \begin{align*}
      \min \cbrac{1 + \frac{t_k}{d_0}, 1 + \frac{\mu}{\lambda_k}} \cdot \rbrac{f(x_{k+1}) - f(x_\star^f)} \leq f(x_k) - f(x_\star^f),
   \end{align*}
   or equivalently,
   \begin{align}
      \label{eq:perstep-favorable-descent}
      f(x_{k+1}) - f(x_\star^f) \leq \rbrac{1 + \min\cbrac{\frac{t_k}{d_0}, \frac{\mu}{\lambda_k}}}^{-1} \cdot \rbrac{f(x_k) - f(x_\star^f)}.
   \end{align}
   Unrolling this inequality for $k=0, \hdots, K-1$ gives 
   \begin{align*}
      f(x_{K}) - f(x_\star^f) \leq \prod_{k=0}^{K-1} \rbrac{1 + \min\cbrac{\frac{t_k}{d_0}, \frac{\mu}{\lambda_k}}}^{-1} \cdot \rbrac{f(x_0) - f(x_\star^f)}.
   \end{align*}
\end{proof}

\subsection{\texorpdfstring{Proof of \Cref{lemma:variant}}{Proof of Lemma~\ref{lemma:variant}}}
\begin{proof}
   Let 
   \begin{align*}
      u \in \argmin_{z \in \gB_{t}(x)} \cbrac{f(z) + \frac{\lambda}{2}\norm{z - x}^2}.
   \end{align*}
   Applying \Cref{fact:subdiff-calculus}, the first order optimality condition for this problem is
   \begin{align*}
      0 \in \partial f(u) + \lambda \rbrac{u - x} + \gN_{\gB_{t}(x)}(u),
   \end{align*}
   where $\gN_{\gB_{t}(x)}(u)$ is the normal cone of the ball centered at $x$ with radius $t$ associated with point $u$.
   We now distinguish two cases. 
   \begin{itemize}
      \item \textbf{Interior case.} If $u$ lies in the interior of the ball, then $\norm{u - x} < t$ and $\gN_{\gB_t(x)}(u) = \cbrac{0}$.
      The optimality condition reduces to 
      \begin{align*}
         0 \in \partial f(u) + \lambda(u - x),
      \end{align*}
      which is equivalent to $\lambda(x - u) \in \partial f(u)$.
      In particular, when $\lambda = 0$, this yields $0 \in \partial f(u)$, recovering the interior case of the second brox theorem.

      \item \textbf{Boundary case.} If $u$ lies on the boundary of the ball, then $\norm{u - x} = t$.
      It is known that $\gN_{\gB_t(x)}(u) = \cbrac{\alpha(u - x) : \alpha \geq 0}$.
      This implies 
      \begin{align*}
         \rbrac{c_t(x) + \lambda}(x - u) \in \partial f(u),
      \end{align*}
      for some $c_t(x) \geq 0$.
   \end{itemize} 
   Combining the two cases, we conclude that there exists some $c_t(x) \ge 0$ for which $\rbrac{c_t(x) + \lambda}(x - u) \in \partial f(u)$.  
   Finally, by the definition of a subgradient, it follows that
   \begin{align*}
      f(y) - f(u) \geq \rbrac{c_t(x) + \lambda}\inner{x - u}{y - u}.
   \end{align*}
\end{proof}

\subsection{\texorpdfstring{Proof of \Cref{lemma:bounded-range}}{Proof of Lemma~\ref{lemma:bounded-range}}}

\begin{proof}
   By the optimality condition associated with $\prox{\nicefrac{f}{\lambda}}{x}$, we have
   \begin{align*}
      f(\prox{\nicefrac{f}{\lambda}}{x}) + \frac{\lambda}{2}\norm{x - \prox{\nicefrac{f}{\lambda}}{x}}^2 \leq f(x_\star^f) + \frac{\lambda}{2}\norm{x - x_\star^f}^2 \qquad \forall \lambda \in (0, \infty).
   \end{align*}
   Since $f(\prox{\nicefrac{f}{\lambda}}{x}) \geq f(x_\star^f)$, it follows immediately that
   \begin{align*}
      \norm{x - \prox{\nicefrac{f}{\lambda}}{x}}^2 \leq \norm{x - x_\star^f}^2.
   \end{align*}
   In the special case $\lambda = 0$, the proximal operator may be set-valued, and we adopt the convention $\prox{\nicefrac{f}{\lambda}}{x} \eqdef \ProjOn{\gX_\star^f}{x}$.
   Then $\norm{x - \prox{\nicefrac{f}{\lambda}}{x}} = \dist{x}{\gX_\star^f} \leq \norm{x - x_\star^f}$ for any $x_\star^f \in \gX_\star^f$.
   Finally, in the case $\lambda = \infty$, we have $\norm{x - \prox{\nicefrac{f}{\lambda}}{x}} = 0$.
\end{proof}

\subsection{\texorpdfstring{Proof of \Cref{lemma:monotonicity-displacement}}{Proof of Lemma~\ref{lemma:monotonicity-displacement}}}

\begin{proof}
   Pick $\lambda_1 > \lambda_2 > 0$ and let 
   \begin{align*}
      u_1 \eqdef \prox{\nicefrac{f}{\lambda_1}}{x}, \qquad u_2 \eqdef \prox{\nicefrac{f}{\lambda_2}}{x}.
   \end{align*}
   We have 
   \begin{align*}
      \lambda_1 \rbrac{x - u_1} \in \partial f(u_1), \qquad \lambda_2 \rbrac{x - u_2} \in \partial f(u_2).
   \end{align*}
   Since according to \Cref{fact:mono}, $\partial f$ is monotone, it follows that
   \begin{align*}
      \inner{\lambda_1(x - u_1) - \lambda_2(x - u_2)}{u_1 - u_2} \geq 0.
   \end{align*}
   Now, denote $d_1 \eqdef x - u_1$, $d_2 \eqdef x - u_2$.
   Since $u_1 - u_2 = x - d_1 - x + d_2 = d_2 - d_1$, we have 
   \begin{align*}
      \inner{\lambda_1d_1 - \lambda_2d_2}{d_2 - d_1} \geq 0,
   \end{align*}
   i.e., 
   \begin{align*}
      \rbrac{\lambda_1 + \lambda_2}\inner{d_1}{d_2} - \lambda_1 \norm{d_1}^2 - \lambda_2 \norm{d_2}^2 \geq 0.
   \end{align*}
   Thus, 
   \begin{align}
      \label{eq:ineq-tmp}
      \lambda_1\norm{d_1}^2 + \lambda_2\norm{d_2}^2 \leq \rbrac{\lambda_1 + \lambda_2}\inner{d_1}{d_2} \leq \rbrac{\lambda_1 + \lambda_2}\norm{d_1}\norm{d_2}.
   \end{align}
   We want to prove that $\norm{d_1} \leq \norm{d_2}$.
   If $x \in \gX_\star^f$, then $\norm{d_1} = \norm{d_2}$ and the conclusion follows trivially.
   Now consider $x \notin \gX_\star^f$, where we have $\norm{d_1}, \norm{d_2} > 0$.
   Suppose for a contradiction that $\norm{d_1} > \norm{d_2}$ and let $r \eqdef \frac{\norm{d_1}}{\norm{d_2}} > 1$.
   Dividing both sides of \cref{eq:ineq-tmp} by $\norm{d_2}^2$, we have 
   \begin{align*}
      \lambda_1 r^2 - (\lambda_1 + \lambda_2)r + \lambda_2 \leq 0.
   \end{align*}
   However, this is only possible when $r \in \sbrac{\frac{\lambda_2}{\lambda_1}, 1}$.
   The contradiction implies that $\norm{d_1} \leq \norm{d_2}$.
   The case of $\lambda_2 = 0$ is trivial.
   Thus, we conclude that for any fixed $x$, the displacement function $\phi(\lambda, x)$ is a non-increasing function of $\lambda \geq 0$.
\end{proof}

\subsection{\texorpdfstring{Proof of \Cref{lemma:limit}}{Proof of Lemma~\ref{lemma:limit}}}
\begin{proof}
   If $x \in \gX_f^\star$, then $\phi(\lambda, x) = 0$ and the claim holds trivially.
   Now assume $x \notin \gX_\star^f$.
   Denoting $u_\lambda = \prox{\nicefrac{f}{\lambda}}{x}$, we have 
   \begin{align*}
      f(u_\lambda) + \frac{\lambda}{2}\norm{u_\lambda - x}^2 \leq f(x).
   \end{align*}
   This gives us 
   \begin{align*}
      \norm{u_\lambda - x}^2 \leq \frac{2}{\lambda}\rbrac{f(x) - f_{\inf}}
   \end{align*}
   Hence, as long as $\lambda > \frac{2\rbrac{f(x) - f_{\inf}}}{\varepsilon^2}$, we have $\phi(\lambda, x) = \norm{u_\lambda - x} < \varepsilon$, proving the first part of the claim.
   
   Now, consider a sequence $\cbrac{\lambda_n}^{\infty}_{n=1}$, where $\lim \limits_{n \rightarrow \infty} \lambda_n = 0$, $\lambda_n > 0$
   Optimality of $u_{\lambda_n}$ gives 
   \begin{align}
      \label{eq:temp:key-1}
      f(u_{\lambda_n}) + \frac{\lambda_n}{2}\norm{u_{\lambda_n} - x}^2 \leq f(x_\star^f) + \frac{\lambda_n}{2}\norm{x_\star^f - x}^2 \qquad \forall x_\star^f \in \gX_\star^f.
   \end{align}
   This gives 
   \begin{align*}
      f(u_{\lambda_n}) \leq f(x_\star^f) + \frac{\lambda_n}{2}\norm{x_\star^f - x}^2 \qquad \forall x_\star^f \in \gX_\star^f.
   \end{align*}
   Thus, $\lim \limits_{n \rightarrow \infty} f(u_{\lambda_n}) \leq f_{\inf}$,
   and since $f(u_{\lambda_n}) \geq f_{\inf}$, we have $\lim \limits_{n \rightarrow \infty} f(u_{\lambda_n}) = f_{\inf}$.
   This tells us every cluster point of $\cbrac{u_{\lambda_n}}_{n=0}^\infty$ lies in $\gX_f^\star$.

   From \Cref{lemma:bounded-range}, we know that $\norm{u_{\lambda_n} - x} \leq \norm{x - x_\star^f}$ for any $x_\star^f$ .
   Thus $\limsup \limits_{n \rightarrow \infty} \norm{u_{\lambda_n} - x} \leq \dist{x}{\gX_f^\star}$.
   
   Suppose there exists a subsequence $\cbrac{u_{\lambda_{n_k}}}$ such that 
   \begin{align}
      \label{eq:badt}
      \lim \limits_{k \rightarrow \infty} \norm{x - u_{\lambda_{n_k}}} < \dist{x}{\gX_f^\star}.
   \end{align}
   Since we know $\lim_{n \rightarrow \infty} f(u_{\lambda_n}) = f_{\inf}$ and the sequence is bounded, some sequence converges to a point $\bar{u} \in \R^d$. By lower-semicontinuity we have $f(\bar{u}) \leq \liminf \limits_{k \rightarrow \infty} f(u_{\lambda_{n_k}}) = f_{\inf}$,
   which indicates $\bar{u} \in \gX_\star^f$, contradicting \eqref{eq:badt}. 
   Thus $\liminf \limits_{n \rightarrow \infty} \norm{x - u_{\lambda_n}} \geq \dist{x}{\gX_\star^f}$.
   
   Combining both we conclude that $\lim_{n \rightarrow \infty} \norm{x - u_{\lambda_n}} = \dist{x}{\gX_\star^f}$.
   Since the sequence $\cbrac{\lambda_n}$ is arbitrary, we conclude that $\lim \limits_{\lambda \rightarrow 0^+ } \norm{x - \prox{\nicefrac{f}{\lambda}}{x}} = \dist{x}{\gX_f^\star}$, which also means that $\phi(\lambda, x)$ is continuous at $\lambda=0$.
\end{proof}

\subsection{\texorpdfstring{Proof of \Cref{lemma:cont-distlambda}}{Proof of Lemma~\ref{lemma:cont-distlambda}}}
\begin{proof}
   Continuity at $\lambda = 0$ is established in the proof of \Cref{lemma:limit}.
   We therefore consider $\lambda > 0$.
   Consider a sequence $\cbrac{\lambda_n}_{n=0}^{\infty}$ such that $\lim \limits_{n \rightarrow \infty} \lambda_n = \lambda$, denote $u_n = \prox{\nicefrac{f}{\lambda_n}}{x}$, and $u = \prox{\nicefrac{f}{\lambda}}{x}$.
   From \Cref{lemma:bounded-range} we know that $\norm{u_n - x} \leq \dist{x}{\gX_f^\star}$, i.e., $\cbrac{u_n}$ a bounded sequence.
   By the Bolzano-Weierstrass theorem, it admits a convergent subsequence $\cbrac{u_{n_k}}$ with limit $\bar{u} \in \R^d$.
   By the definition of the proximal operator, for each $n$
   \begin{align*}
      0 \in \partial f(u_n) + \lambda_n(u_n - x).
   \end{align*}
   Define $v_n \eqdef - \lambda_n(u_n - x)$, so $v_n \in \partial f(u_n)$.
   Since 
   \begin{align*}
      u_{n_k} \rightarrow \bar{u}, \quad \lambda_{n_k} \rightarrow \lambda,
   \end{align*}
   we have 
   \begin{align*}
      v_{n_k} = -\lambda_{n_k} (u_{n_k} - x) \rightarrow -\lambda(\bar{u} - x) =: \bar{v}.
   \end{align*}
   Since $f$ is proper, closed and convex, $\partial f$ is maximally monotone \citep[Theorem 20.40]{bauschke2020correction} and its graph is closed.
   Thus, from 
   \begin{align*}
      (u_{n_k}, v_{n_k}) \rightarrow (\bar{u}, \bar{v}), \quad v_{n_k} \in \partial f(u_{n_k}),
   \end{align*}
   we deduce that $\bar{v} \in \partial f(\bar{u})$, i.e., 
   \begin{align*}
      0 \in \partial f(\bar{u}) + \lambda(\bar{u} - x).
   \end{align*}
   This shows that $\bar{u}$ minimizes the regularized objective.
   Since the minimizer is unique, we conclude that $\bar{u} = u$.
   We have therefore shown that every convergent subsequence of $\cbrac{u_n}$ converges to $u$.
   Since it is bounded, $\cbrac{u_n}$ converges to $u$.
   Now since $u_n \rightarrow u$ and the map $y \mapsto \norm{x - y}$ is continuous, 
   \begin{align*}
      \phi(\lambda_n, x) = \norm{x - u_n} \rightarrow \norm{x - u} = \phi(\lambda, x).
   \end{align*}
   Thus, $\phi(\lambda, x)$ is continuous at every $\lambda > 0$.
   
   We conclude the proof combining the above two cases.
\end{proof}

\subsection{\texorpdfstring{Proof of \Cref{lemma:phi-PPM-sequence}}{Proof of Lemma~\ref{lemma:phi-PPM-sequence}}}
\begin{proof}
   Due to the first order optimality condition,
   \begin{align*}
      \lambda(x_k - x_{k+1}) \in \partial f(x_{k+1}), \quad \lambda(x_{k+1} - x_{k+2}) \in \partial f(x_{k+2}).
   \end{align*}
   Since $\partial f$ is monotone, we have 
   \begin{align*}
      \inner{\lambda(x_k - x_{k+1}) - \lambda(x_{k+1} - x_{k+2})}{x_{k+1} - x_{k+2}} \geq 0,
   \end{align*}
   In the case of $\lambda = 0$, we have $x_{k} = x_{k+1} = x_{k+2}$, and the comparison is trivial.
   When $\lambda > 0$, we divide by $\lambda$ and obtain 
   \begin{align*}
      \inner{x_k - x_{k+1}}{x_{k+1} - x_{k+2}} \geq \norm{x_{k+1} - x_{k+2}}^2.
   \end{align*}
   In addition,
   \begin{align*}
      \inner{x_k - x_{k+1}}{x_{k+1} - x_{k+2}} \leq \norm{x_k - x_{k+1}}\norm{x_{k+1} - x_{k+2}},
   \end{align*}
   due to the Cauchy-Schwarz inequality. Now, if $\norm{x_{k+1} - x_{k+2}} = 0$, the conclusion trivially holds. Otherwise, we divide both sides by this term after chaining the inequality to obtain
   \begin{align*}
      \norm{x_{k+1} - x_{k+2}} \leq \norm{x_k - x_{k+1}}.
   \end{align*}
\end{proof}

\subsection{\texorpdfstring{Proof of \Cref{lemma:Lipschitz}}{Proof of Lemma~\ref{lemma:Lipschitz}}}
\begin{proof}
   The proof follows immediately from the fact that the proximal operator of a proper, closed and convex function is firmly nonexpansive.
\end{proof}

\subsection{\texorpdfstring{Proof of \Cref{lemma:cont-distx}}{Proof of Lemma~\ref{lemma:cont-distx}}}
\begin{proof}
   When $\lambda = 0$, $\phi(\lambda, x) = \dist{x}{\gX_f^\star}$, which is continuous.
   When $\lambda > 0$, for any $x, y \in \dom f$
   \begin{align*}
      \abss{\norm{x - \prox{\nicefrac{f}{\lambda}}{x}} - \norm{y - \prox{\nicefrac{f}{\lambda}}{y}}} &\leq \norm{x - \prox{\nicefrac{f}{\lambda}}{x} - y + \prox{\nicefrac{f}{\lambda}}{y}} \\
      &\leq \norm{x - y} + \norm{\prox{\nicefrac{f}{\lambda}}{x} - \prox{\nicefrac{f}{\lambda}}{y}} \\
      &\leq 2\norm{x - y},
   \end{align*}
   where the last inequality is a result of applying \Cref{lemma:Lipschitz}.
   Thus, $\phi(\lambda, x)$ is 2-Lipschitz, and hence continuous.
\end{proof}

\subsection{\texorpdfstring{Proof of \Cref{lemma:ettet}}{Proof of Lemma~\ref{lemma:ettet}}}
\begin{proof}
   We prove the contrapositive of the statement.
   Assume that $\norm{x_k - \prox{\nicefrac{f}{\lambda_k}}{x_k}} < t$.
   By the optimality condition of the proximal subproblem, for any $x_\star^f \in \gX_\star^f$,
   \begin{align*}
      f(\prox{\nicefrac{f}{\lambda_k}}{x_k}) + \frac{\lambda_k}{2}\norm{\prox{\nicefrac{f}{\lambda_k}}{x_k} - x_k}^2 \leq f(x_\star^f) + \frac{\lambda_k}{2}\norm{x_k - x_\star^f}^2.
   \end{align*}
   Rearranging terms gives
   \begin{align*}
      f(\prox{\nicefrac{f}{\lambda_k}}{x_k}) - f(x_\star^f) \leq \frac{\lambda_k}{2}\rbrac{\norm{x_k - x_\star^f}^2 - \norm{\prox{\nicefrac{f}{\lambda_k}}{x_k} - x_k}^2}.
   \end{align*}
   Since the inequality holds for arbitrary $x_\star^f \in \gX_\star^f$, we obtain
   \begin{align}
      \label{eq:keyalpha1-q}
      f(\prox{\nicefrac{f}{\lambda_k}}{x_k}) - f_{\inf}\leq \frac{\lambda_k}{2}\rbrac{\distsq{x_k}{\gX_\star^f} - \norm{\prox{\nicefrac{f}{\lambda_k}}{x_k} - x_k}^2}.
   \end{align}
   On the other hand, invoking \Cref{assmp:sharp-minima}, we have
   \begin{align*}
      f(\prox{\nicefrac{f}{\lambda_k}}{x_k}) - f_{\inf}\geq \alpha \cdot \distpow{q}{\prox{\nicefrac{f}{\lambda_k}}{x_k}}{\gX_\star^f}.
   \end{align*}
   Applying the triangle inequality yields
   \begin{align}
      \label{eq:keyalpha2-q}
      f(\prox{\nicefrac{f}{\lambda_k}}{x_k}) - f_{\inf} &\geq \alpha \rbrac{\norm{x_k - \ProjOn{\gX_\star^f}{\prox{\nicefrac{f}{\lambda_k}}{x_k}}} - \norm{x_k - \prox{\nicefrac{f}{\lambda_k}}{x_k}}}^q \notag \\
      &= \alpha \rbrac{\dist{x_k}{\gX_\star^f} - \norm{x_k - \prox{\nicefrac{f}{\lambda_k}}{x_k}}}^q.
   \end{align}
   Combining \cref{eq:keyalpha1-q} and \cref{eq:keyalpha2-q}, we obtain
   \begin{align*}
      \alpha \rbrac{\dist{x_k}{\gX_\star^f} - \norm{x_k - \prox{\nicefrac{f}{\lambda_k}}{x_k}}}^q \leq \frac{\lambda_k}{2}\rbrac{\distsq{x_k}{\gX_\star^f} - \norm{\prox{\nicefrac{f}{\lambda_k}}{x_k} - x_k}^2}.
   \end{align*}
   Using the factorization
   \begin{align*}
      &\distsq{x_k}{\gX_\star^f} - \norm{\prox{\nicefrac{f}{\lambda_k}}{x_k} - x_k}^2 \\
      &\quad =
      \rbrac{\dist{x_k}{\gX_\star^f} - \norm{x_k - \prox{\nicefrac{f}{\lambda_k}}{x_k}}} \rbrac{\dist{x_k}{\gX_\star^f} + \norm{x_k - \prox{\nicefrac{f}{\lambda_k}}{x_k}}},
   \end{align*}
   and recalling that $\dist{x_k}{\gX_\star^f} \ge t$ and
   $\norm{x_k - \prox{\nicefrac{f}{\lambda_k}}{x_k}} < t$, we may divide both sides of the inequality by $\dist{x_k}{\gX_\star^f} - \norm{x_k - \prox{\nicefrac{f}{\lambda_k}}{x_k}} > 0$ to obtain
   \begin{align*}
      \alpha \rbrac{\dist{x_k}{\gX_\star^f} - \norm{x_k - \prox{\nicefrac{f}{\lambda_k}}{x_k}}}^{q-1} \leq \frac{\lambda_k}{2} \rbrac{\dist{x_k}{\gX_\star^f} + \norm{x_k - \prox{\nicefrac{f}{\lambda_k}}{x_k}}}.
   \end{align*}
   Using again $\dist{x_k}{\gX_\star^f} \geq t$ and
   $\norm{x_k - \prox{\nicefrac{f}{\lambda_k}}{x_k}} < t$, we further obtain
   \begin{align*}
      \alpha \rbrac{\dist{x_k}{\gX_\star^f} - t}^{q-1} < \frac{\lambda_k}{2} \rbrac{\dist{x_k}{\gX_\star^f} + t}.
   \end{align*}
   Rearranging yields
   \begin{align*}
      \lambda_k > \frac{2\alpha \rbrac{\dist{x_k}{\gX_\star^f} - t}^{q-1}} {\dist{x_k}{\gX_\star^f} + t},
   \end{align*}
   which concludes the proof.
\end{proof}

\subsection{\texorpdfstring{Proof of \Cref{lemma:explaining}}{Proof of Lemma~\ref{lemma:explaining}}}
\begin{proof}
   Since $x_\star^f \notin \gB_t(x)$, the minimizer of $f$ over the ball $\gB_t(x)$ must lie on the boundary $\bdry \gB_t(x)$ due to \Cref{fact:first-brox}.
   Let us denote this minimizer by $u_1$.
   Then
   \begin{align*}
      u_1 &= \argmin_{z \in \bdry \gB_t(x)} \cbrac{f(z)} = \argmin_{z \in \bdry \gB_t(x)} \cbrac{f(z) + \frac{\lambda}{2}\norm{z - x}^2}
   \end{align*}
   because on $\bdry \gB_t(x)$, the term $\norm{z-x}^2$ is constant and equal to $t^2$.
   Similarly, since the unconstrained minimizer of $f(z) + \frac{\lambda}{2}\norm{z - x}^2$ lies outside $\gB_t(x)$ and the objective is convex, its constrained minimizer over $\gB_t(x)$ must also lie on the boundary.
   Denoting this minimizer by $u_2$, we have
   \begin{align*}
      u_2 = \argmin_{z \in \bdry \gB_t(x)}\cbrac{f(z) + \frac{\lambda}{2}\norm{z - x}^2} = u_1.
   \end{align*}
\end{proof}


\end{document}